\documentclass[a4paper, 10pt]{article}
\usepackage{amsmath}
\usepackage{amssymb,esint}
\usepackage{amscd}
\usepackage{xspace}
\usepackage{fancyhdr}
\newtheorem{theorem}{Theorem}[section]

\newtheorem{lemma}[theorem]{Lemma}

\newtheorem{proposition}[theorem]{Proposition}
\newtheorem{remark}[theorem]{Remark}

\newenvironment{proof}[1][Proof]{\textbf{#1.} }{\hfill\rule{0.5em}{0.5em}}
{\catcode`\@=11\global\let\AddToReset=\@addtoreset
\AddToReset{equation}{section}

\AddToReset{theorem}{section}

\begin{document}
\title{Global estimates for quasilinear parabolic equations on Reifenberg flat domains and its applications to  Riccati type parabolic equations with distributional data}
\author{
 {\bf Quoc-Hung Nguyen\thanks{ E-mail address: quoc-hung.nguyen@epfl.ch;~ Hung.Nguyen-Quoc@lmpt.univ-tours.fr}}\\[0.5mm]
{\small \'Ecole Polytechnique F\'ed\'erale de Lausanne }\\
{\small  SB MATHAA CAMA, Station 8, CH-1015 Lausanne, Switzerland}} 
\date{November 21, 2014}  
\maketitle
\begin{abstract}
In this paper,  we prove global weighted Lorentz and Lorentz-Morrey estimates for gradients of solutions to the quasilinear parabolic equations:
$$u_t-\operatorname{div}(A(x,t,\nabla u))=\operatorname{div}(F),$$
 in a bounded domain $\Omega\times (0,T)\subset\mathbb{R}^{N+1}$, under minimal regularity assumptions on the boundary of domain and on nonlinearity $A$. Then results yields existence of a solution to the  Riccati type parabolic equations:
 $$u_t-\operatorname{div}(A(x,t,\nabla u))=|\nabla u|^q+\operatorname{div}(F)+\mu,$$ 
where $q>1$ and $\mu$ is a bounded Radon measure.
\end{abstract}   
                       
 \section{Introduction and main results} 
 In this article, we are concerned with the global weighted Lorentz space estimates for gradients of  weak solutions to quasilinear parabolic equations in divergence form:
 \begin{equation}\label{5hh070120148}
                       \left\{
                                       \begin{array}
                                       [c]{l}%
                                       {u_{t}}-\operatorname{div}(A(x,t,\nabla u))=\operatorname{
                                       div}(F)~~\text{in }\Omega_T,\\ 
                         u=0~~~~~~~\text{on}~~
                                                                                              \partial_p(\Omega \times (0,T)),
                                                                                                 \\                          
                                       \end{array}
                                       \right.  
                                       \end{equation}   where  $\Omega_T:=\Omega\times (0,T)$ is a bounded open subset of $\mathbb{R}^{N+1}$, $N\geq2$, $ \partial_p(\Omega \times (0,T))=(\partial\Omega\times(0,T))\cup (\Omega\times\{t=0\})$,   $F\in L^1(\Omega_T,\mathbb{R}^N)$ is a given vector field and  the nonlinearity  $A:\mathbb{R}^N\times\mathbb{R}\times \mathbb{R}^N\to \mathbb{R}^N$ is a Carath\'eodory vector valued function, i.e. $A$ is measurable in $(x,t)$ and continuous with respect to $\nabla u$ for a.e. $(x,t)$.\\
                                       We suppose in this paper that $A$ satisfies 
                                       \begin{align}
                                       \label{5hhconda}&|A(x,t,\zeta)|\le \Lambda_1 |\zeta|,
                                       \end{align}
                                       and 
                                       \begin{align}
                                       \label{5hhcondb}
                                       \left \langle A(x,t,\zeta)-A(x,t,\xi) ,\zeta-\xi\right\rangle\geq \Lambda_2 |\zeta-\xi|^2,
                                       \end{align}
                                          for every $(\xi,\zeta)\in \mathbb{R}^N\times \mathbb{R}^N$ and a.e. $(x,t)\in \mathbb{R}^N\times \mathbb{R}$, where  $\Lambda_1$ and $\Lambda_2$ are positive constants. In addtion, we also assume that the derivatives of $A$ with respect to $\zeta$ are bounded, that is, 
                                          \begin{align}\label{181120141}
                                          |A_\zeta(x,t,\zeta)|\leq \Lambda_1, 
                                          \end{align}      for any $\zeta\in \mathbb{R}^N$ and $(x,t)\in \mathbb{R}^N$. We remark that the condition \eqref{181120141} is needed in order to ensure that the reference problems \eqref{5hheq4} and \eqref{5hh1610134} in the next section have $C^{0,1}$ regularity solutions (see \cite{55Li1,55Li2}), which will be used in the sequel. \\          
Throughout the paper, we assume that $A$ satisfies \eqref{5hhconda} and \eqref{5hhcondb}, \eqref{181120141}. Besides, we always denote  $T_0=\text{diam}(\Omega)+T^{1/2}$ and $Q_\rho(x,t)=B_\rho(x)\times (t-\rho^2,t)$ $\tilde{Q}_\rho(x,t)=B_\rho(x)\times (t-\rho^2/2,t+\rho^2/2)$ for $(x,t)\in\mathbb{R}^{N+1}$ and $\rho>0$.\\
A weak solution $u$ of \eqref{5hh070120148} is understood in the standard weak (distributional) sense, that is $u\in L^1(0,T,W_0^{1,1}(\Omega))$ is a weak solution of \eqref{5hh070120148} if 
\begin{align*}
 -\int_{\Omega_T}u \varphi_tdxdt+\int_{\Omega_T}A(x,t,\nabla u)\nabla \varphi dxdt= -\int_{\Omega_T}F\nabla\varphi dxdt
 \end{align*}
 for all $\varphi \in C_c^1([0,T)\times \Omega)$.                               
The existence and uniqueness of weak solutions in $L^2(0,T,H_0^1(\Omega))$ to problem \eqref{5hh070120148} with $F\in L^2(\Omega_T,\mathbb{R}^N)$ is given at the beginning of the next section.\medskip                                            \\          
      For our purpose, we need a condition on $\Omega$ which is expressed in the following way. We say that $\Omega$ is a $(\delta,R_0)-$Reifenberg flat domain for $\delta\in (0,1)$ and $R_0>0$ if for every $x\in\partial \Omega$ and every $r\in(0,R_0]$, there exists a system of coordinates $\{z_1,z_2,...,z_n\}$, which may depend on $r$ and $x$, so that in this coordinate system $x=0$ and that 
         \begin{equation}
         B_r(0)\cap \{z_n>\delta r\}\subset B_r(0)\cap \Omega\subset B_r(0)\cap\{z_n>-\delta r\}.
         \end{equation}
         
 We notice that this class of flat domains is rather wide since it includes $C^1$ domains, Lipschitz domains with sufficiently small Lipschitz constants and even fractal domains. Besides, it has many important roles in the theory of minimal surfaces and free boundary problems. This class appeared first in a work of Reifenberg (see \cite{55Re}) in the context of Plateau problem. Its properties can be found in \cite{55KeTo1,55KeTo2,55To}.

     We also require  that the nonlinearity $A$  satisfies a smallness condition of BMO type in the $x$-variable in the sense that $A(x,t,\zeta)$ satisfies a $(\delta,R_0)$-BMO condition for some $\delta, R_0>0$ with exponent $p>0$ if 
                                   \begin{equation*}
                                   [A]^{R_0}_p:=\mathop {\sup }\limits_{(y,s)\in \mathbb{R}^N\times\mathbb{R},0<r\leq R_0}\left(\fint_{Q_r(y,s)}\left(\Theta(A,B_r(y))(x,t)\right)^pdxdt\right)^{\frac{1}{p}} \leq \delta,
                                   \end{equation*}         
 where 
                                   \begin{equation*}
                                   \Theta(A,B_r(y))(x,t):=\mathop {\sup }\limits_{\zeta\in\mathbb{R}^N\backslash\{0\}}\frac{|A(x,t,\zeta)-\overline{A}_{B_r(y)}(t,\zeta)|}{|\zeta|},
                                   \end{equation*}
                                    and $\overline{A}_{B_r(y)}(t,\zeta)$ is denoted the average of $A(t,.,\zeta)$ over the ball $B_r(y)$, i.e,
                                   \begin{equation*}
                                   \overline{A}_{B_r(y)}(t,\zeta):=\fint_{B_r(y)}A(x,t,\zeta)dx=\frac{1}{|B_r(y)|}\int_{B_r(y)}A(x,t,\zeta)dx.
                                   \end{equation*}                                                              
                                                                      
                                      The above condition appeared in our previous paper \cite{55QH2}. It is easy to  see that the $(\delta,R_0)-$BMO  is satisfied when $A$ is continuous or has small jump discontinuities with respect to $x$.
    We recall that a positive function $w\in L^1_{\text{loc}}(\mathbb{R}^{N+1})$ is called an $\mathbf{A}_{p}$ weight, $1\leq p<\infty$ if there holds
    \begin{align*}
    [w]_{\mathbf{A}_{p}}:= \mathop {\sup }\limits_{\tilde{Q}_\rho(x,t)\subset\mathbb{
    R}^{N+1}}\left(\fint_{\tilde{Q}_\rho(x,t)}w(y,s)dyds
    \right)\left(\fint_{\tilde{Q}_\rho(x,t)}w(y,s)^{-\frac{1}{p-1}}dyds
        \right)^{p-1}<\infty~~\text{when }~p>1,
    \end{align*} 
    \begin{align*}
        [w]_{\mathbf{A}_{1}}:= \mathop {\sup }\limits_{\tilde{Q}_\rho(x,t)\subset\mathbb{
        R}^{N+1}}\left(\fint_{\tilde{Q}_\rho(x,t)}w(y,s)dyds
        \right)\mathop {\operatorname{ess}\sup }\limits_{(y,s)\in \tilde{Q}_\rho(x,t)}\frac{1}{w(y,s)}<\infty~~\text{when }~p=1.
        \end{align*}
     The $[w]_{\mathbf{A}_p}$ is called the $\mathbf{A}_{p}$ constant of $w$. \\                                  
   A positive function $w\in L^1_{\text{loc}}(\mathbb{R}^{N+1})$ is called an $\mathbf{A}_{\infty}$ weight if there are two positive constants $C$ and $\nu$ such that
                                $$w(E)\le C \left(\frac{|E|}{|Q|}\right)^\nu w(Q),
                                $$
                                 for all cylinder $Q=\tilde{Q}_\rho(x,t)$ and all measurable subsets $E$ of $Q$. The pair $(C,\nu) $ is called the $\mathbf{A}_\infty$ constant of $w$ and is denoted by $[w]_{\mathbf{A}_\infty}$.  
  It is well known that this class is the union of $\mathbf{A}_p$ for all $p\in (1,\infty)$, see \cite{55Gra}.\\  
   If $w$ is a weight function belonging to $w\in \mathbf{A}_{\infty}$ and $E\subset \mathbb{R}^{N+1}$ a Borel set, $0<q<\infty$, $0<s\leq\infty$, the weighted Lorentz space $L^{q,s}_w(E)$  is the set of measurable functions $g$ on $E$ such that 
        \begin{equation*}
     ||g||_{L^{q,s}_w(E)}:=\left\{ \begin{array}{l}
              \left(q\int_{0}^{\infty}\left(\rho^qw\left(\{(x,t)\in E:|g(x,t)|>\rho\}\right)\right)^{\frac{s}{q}}\frac{d\rho}{\rho}\right)^{1/s}<\infty~\text{ if }~s<\infty, \\ 
               \sup_{\rho>0}\rho \left( w\left(\{(x,t)\in E:|g(x,t)|>\rho\}\right)\right)^{1/q}<\infty~~\text{ if }~s=\infty. \\ 
               \end{array} \right.
        \end{equation*}              
               Here we write $w(O)=\int_{O}w(x,t)dxdt$ for a measurable set $O\subset \mathbb{R}^{N+1}$.  Obviously,  
               $
                       ||g||_{L^{q,q}_w(E)}=||g||_{L^q_w(E)}
                      $,
                      thus $L^{q,q}_w(E)=L^{q}_w(E)$.               As usual, when $w \equiv 1$  we  write simply $L^{q,s}(E)$ instead of $L^{q,s}_w(E)$.                                                     
               In this paper, $\mathcal{M}$ denotes the parabolic Hardy-Littlewood maximal function defined for each locally integrable function  $f$ in $\mathbb{R}^{N+1}$ by
                     \begin{equation*}
                     \mathcal{M}(f)(x,t)=\sup_{\rho>0}\fint_{\tilde{Q}_\rho(x,t)}|f(y,s)|dyds~~\forall (x,t)\in\mathbb{R}^{N+1}.
                     \end{equation*}
            If $p>1$ and $w\in \mathbf{A}_p$ we verify that $\mathcal{M}$ is operator from $L^1(\mathbb{R}^{N+1})$ into $L^{1,\infty}(\mathbb{R}^{N+1})$ and $L^{p,s}_w(\mathbb{R}^{N+1})$  into itself for $0<s\leq \infty$, see  \cite{55Stein2,55Stein3,55Tur}. 
         \\ We would like to mention that the use
            of the Hardy-Littlewood maximal function in non-linear degenerate problems was started in the elliptic setting by T.  Iwaniec in his fundamental paper \cite{Iwa}. \\
                        We now state the main result of the paper.                                      
                \begin{theorem} \label{101120143}Let $F\in L^2(\Omega_T,\mathbb{R}^N)$. There exists a unique weak solution $u\in  L^2(0,T,H_0^1(\Omega))$ of \eqref{5hh070120148}. For any $w\in \mathbf{A}_{\infty}$, $0< q<\infty$, $0<s\leq\infty$ we find  $\delta=\delta(N,\Lambda_1,\Lambda_2,q,s, [w]_{\mathbf{A}_{\infty}})\in (0,1)$ and $s_0=s_0(N,\Lambda_1,\Lambda_2)>0$ such that if $\Omega$ is  $(\delta,R_0)$-Reifenberg flat domain $\Omega$ and $[A]_{s_0}^{R_0}\le \delta$ for some $R_0>0$ then                                       
                  \begin{equation}\label{101120144}
                              ||\mathcal{M}(|\nabla u|^2)||_{L^{q,s}_w(\Omega_T)}\leq C ||\mathcal{M}(|F|^2)||_{L^{q,s}_w(\Omega_T)}.
                                       \end{equation} 
                                        Here $C$ depends only  on $N,\Lambda_1,\Lambda_2,q,s, [w]_{\mathbf{A}_\infty}$ and $T_0/R_0$.                       
                                      \end{theorem}
                                                                                               Since $\mathcal{M}$ is a bounded operator from  $L^{p,s}_w(\mathbb{R}^{N+1})$  into itself for $p>1,0<s\leq \infty$ and $w\in\mathbf{A}_{p}$, thus we obtain the following Theorem.  
        \begin{theorem} \label{101120141} Let $F\in L^2(\Omega_T,\mathbb{R}^N)$ and $s_0$ be in Theorem \ref{101120143}. There exists a unique weak solution $u\in  L^2(0,T,H_0^1(\Omega))$ of \eqref{5hh070120148}. For any $w\in \mathbf{A}_{q/2}$, $2< q<\infty$, $0<s\leq\infty$ we find  $\delta=\delta(N,\Lambda_1,\Lambda_2,q,s, [w]_{\mathbf{A}_{q/2}})\in (0,1)$ and such that if $\Omega$ is  $(\delta,R_0)$-Reifenberg flat domain  and $[A]_{s_0}^{R_0}\le \delta$ for some $R_0>0$ then                          
                                                                                           \begin{equation}\label{101120146}
                                                                                                                                   |||\nabla u|||_{L^{q,s}_w(\Omega_T)}\leq C |||F|||_{L^{q,s}_w(\Omega_T)}.
                                                                                                                                   \end{equation} 
                                                                                                                                    Here $C$ depends only on $N,\Lambda_1,\Lambda_2,q,s, [w]_{\mathbf{A}_{q/2}}$ and $T_0/R_0$.                       
                                                                                                                                  \end{theorem}
 We remark that the global gradient estimates of solutions of \eqref{5hh070120148} obtained in Theorem \ref{101120141} extend results in \cite{55BW1,55BW4,55BOR} to more general nonlinear structure and in the setting of weighted Lorentz spaces. Notice that Theorem \ref{101120143} and  \ref{101120141} in the quasilinear elliptic framework are obtained in \cite{55MePh3}. \\
In the linear case, we obtain global estimates for gradients of weak solutions to problem  
\begin{equation}\label{161120143}
                       \left\{
                                       \begin{array}
                                       [c]{l}%
                                       {u_{t}}-\operatorname{div}(A(x,t,\nabla u))=\operatorname{div}(F)+\mu~~\text{in }\Omega_T,\\ 
                         u=0~~~~~~~\text{on}~~
                                                                                              \partial\Omega \times (0,T),\\
                                                                                              u(0)=\sigma~~~\text{in }~\Omega,
                                                                                                 \\                          
                                       \end{array}
                                       \right.  
                                       \end{equation}    
       where       $F\in L^1(\Omega_T,\mathbb{R}^N)$, $\mu\in\mathfrak{M}_b(\Omega_T)$ the set of bounded Radon measure in $\Omega_T$, $\sigma\in\mathfrak{M}_b(\Omega)$ the set of bounded Radon measure in $\Omega$.                                                                                                          
  \begin{theorem} \label{161120141}  Suppose that $A$ is linear. Let $F\in L^1(\Omega_T,\mathbb{R}^N), \mu\in\mathfrak{M}_b(\Omega_T), \sigma\in\mathfrak{M}_b(\Omega)$, set $\omega=|\mu|+|\sigma|\otimes\delta_{\{t=0\}}$. Let $s_0$ be as in Theorem \ref{101120143}.
   \begin{description}
   \item[a.] For any $q>2, 0<s\leq\infty$, $w\in \mathbf{A}_{q/2}$ and $\mathcal{M}_1[\omega], |F|\in L^{q,s}_w(\Omega_T)$ we find a $\delta=\delta(N,\Lambda_1,\Lambda_2,q,s, [w]_{\mathbf{A}_{q/2}})\in (0,1)$ such that if $\Omega$ is  a $(\delta,R_0)$-Reifenberg flat domain  and $[A]_{s_0}^{R_0}\le \delta$ for some $R_0>0$ there exists a unique weak solution $u\in  L^2(0,T,H_0^{1}(\Omega))$ of \eqref{161120143} and there holds                     
                                                                                                                                      \begin{equation}\label{161120144}
                                                                                                                                      |||\nabla u|||_{L^{q,s}_w(\Omega_T)}\leq C ||\mathcal{M}_1[\omega]||_{L^{q,s}_w(\Omega_T)}+ C |||F|||_{L^{q,s}_w(\Omega_T)},
                                                                                                                                      \end{equation} 
                                                                                                                                       where $C$ depends  only on $N,\Lambda_1,\Lambda_2,q,s, [w]_{\mathbf{A}_{q/2}}$ and $T_0/R_0$.
      \item[b.] For any $\varepsilon\in (0,1), \frac{2(\varepsilon+1)}{\varepsilon+2}<q\leq 2, 0<s\leq \infty$, $w^{2+\varepsilon}\in \mathbf{A}_{1}$ and $\mathcal{M}_1[\omega], |F|\in L^{q,s}_w(\Omega_T)$ we find a $\delta=\delta(N,\Lambda_1,\Lambda_2,q,s,\varepsilon, [w^{2+\varepsilon}]_{\mathbf{A}_{1}})\in (0,1)$ such that if $\Omega$ is a $(\delta,R_0)$-flat domain  and $[A]_{s_0}^{R_0}\le \delta$ for some $R_0>0$ there exists a unique weak solution $u\in  L^{\frac{2(\varepsilon+1)}{\varepsilon+2}}(0,T,W_0^{1,\frac{2(\varepsilon+1)}{\varepsilon+2}}(\Omega))$ of \eqref{161120143} and there holds                            
                                                                                                                                         \begin{equation}\label{161120142}
                                                                                                                                         |||\nabla u|||_{L^{q,s}_w(\Omega_T)}\leq C ||\mathcal{M}_1[\omega]||_{L^{q,s}_w(\Omega_T)}+ C |||F|||_{L^{q,s}_w(\Omega_T)},
                                                                                                                                         \end{equation} 
                                                                                                                          where $C$ depends only  on $N,\Lambda_1,\Lambda_2,q,s,\varepsilon, [w^{2+\varepsilon}]_{\mathbf{A}_{1}}$ and $T_0/R_0$.
                       
   \end{description}                  
                                                                                                                                  \end{theorem}
                                                                                                                                  In above Theorem,  $\mathcal{M}_{1}$ denotes the first order
                                                                                                                                                                 fractional Maximal parabolic  potential  on $\mathbb{R}^{N+1}$ of a positive Radon measure in $\mathbb{R}^{N+1}$ by                 
\begin{equation*}
\mathcal{M}_{1}[\mu](x,t)=\sup_{0<\rho<R}\frac{\mu(\tilde{Q}_\rho(x,t))}{\rho^{N+1}}~~~\forall (x,t)\in \mathbb{R}^{N+1}.\end{equation*}                                                                              
   We can use estimates \eqref{101120146} in Theorem \ref{101120141} and \eqref{161120144}-\eqref{161120142} in Theorem  \ref{161120141} and the following Lemma to get upper bounds for gradients of the solutions in Lorentz-Morrey spaces.
  \begin{lemma}\label{1611201410} Let $0<q<\infty,0<s\leq \infty, \gamma\geq 1$ and $H_1,H_2$ be measurable functions in $\Omega_T$. If 
  \begin{align*}
  ||H_1||_{L^{q,s}_w(\Omega_T)}\leq C(N,q,s,[w^{\gamma}]_{\mathbf{A}_1}) ||H_2||_{L^{q,s}_w(\Omega_T)},
  \end{align*} 
  for any $w^{\gamma}\in \mathbf{A}_1$, then for any $\kappa\in \left(\frac{(N+2)(\gamma-1)}{\gamma},N+2\right],\vartheta\in \left(\frac{N(\gamma-1)}{\gamma}, N\right],$
  \begin{align}\label{1611201411}
    ||H_1||_{L^{q,s;\kappa}_{*}(\Omega_T)}\leq C(N,q,s,\gamma,\kappa) ||H_2||_{L^{q,s;\kappa}_{*}(\Omega_T)},
    \end{align}
    and 
    \begin{align}\label{1611201412}
        ||H_1||_{L^{q,s;\vartheta}_{**}(\Omega_T)}\leq C(N,q,s,\gamma,\vartheta) ||H_2||_{L^{q,s;\vartheta}_{**}(\Omega_T)}.
        \end{align}
  \end{lemma} 
  In \eqref{1611201411}, $L^{q,s;\kappa}_{*}(\Omega_T)$ denotes  Lorentz-Morrey space, is the set of measurable functions $g$ in $\Omega_T$ such that 
              \begin{equation*}                                                              ||g||_{L_{*}^{q,s;\kappa}(\Omega_T)}:=\sup_{0<\rho< T_0, (x,t)\in \Omega_T}\rho^{\frac{\kappa-N-2}{q}}||g||_{L^{q,s}(\tilde{Q}_\rho(x,t)\cap \Omega_T)}<\infty.
                                                                                                                 \end{equation*}
 In \eqref{1611201412}, $L_{**}^{q,s;\vartheta}(\Omega_T)$ is  the Lorentz-Morrey space of measurable functions $g$ in $\Omega_T$ such that 
                                                                     \begin{equation*}
                                                                     ||g||_{L_{**}^{q,s;\vartheta}(\Omega_T)}:=\sup_{0<\rho< \text{diam}(\Omega), x\in \Omega}\rho^{\frac{\vartheta-N}{q}}||g||_{L^{q,s}((B_\rho(x)\cap \Omega)\times (0,T))}<\infty.
                                                                     \end{equation*}
        This Lemma is inspired by \cite[Proof of Theorem 2.3]{55MePh1},  its proof can be found in  \cite[Proof of Theorem 2.21]{55QH2} and  notice that 
        for $(x_0,t_0)\in\Omega_T$ and $0<\rho<T_0$ \begin{align*}
           & w_1(x,t)=\min\{\rho^{-N-2+\kappa-\kappa_1},\max\{|x-x_0|,\sqrt{2|t-t_0|}\}^{-N-2+\kappa-\kappa_1}\},\\&
            w_2(x,t)=\min\{\rho^{-N+\vartheta-\vartheta_1},|x-x_0|^{-N+\vartheta-\vartheta_1}\},
                     \end{align*}
           where $0<\kappa_1<\kappa-\frac{(N+2)(\gamma-1)}{\gamma},$ $ 0<\vartheta_1<\kappa-\frac{N(\gamma-1)}{\gamma}$ and $$[w_1^{\gamma}]_{\mathbf{A}_1}\leq C(N,\kappa_1,\kappa,\gamma),~~[w_2^{\gamma}]_{\mathbf{A}_1}\leq C(N,\vartheta_1,\vartheta,\gamma).$$
   For example, from \eqref{161120144} in Theorem \ref{161120141} and  Lemma \ref{1611201410} we obtain  for $2< q<\infty$, $0<s\leq\infty$ and $0<\kappa\leq N+2, 0<\vartheta\leq N+2$ there hold
   \begin{align}\nonumber
      & |||\nabla u|||_{L_{*}^{q,s;\kappa}(\Omega_T)}\leq C ||\mathcal{M}_1[|\omega|]||_{L_{*}^{q,s;\kappa}(\Omega_T)}+ C |||F|||_{L_{*}^{q,s;\kappa}(\Omega_T)},\\& 
    |||\nabla u|||_{L_{**}^{q,s;\vartheta}(\Omega_T)}\leq C ||\mathcal{M}_1[\omega]||_{L_{**}^{q,s;\vartheta}(\Omega_T)}+C |||F|||_{L_{**}^{q,s;\vartheta}(\Omega_T)}, \label{191120141}
   \end{align}
  and from   \eqref{161120142} in Theorem \ref{161120141} we  also have preceding estimates with $1<q\leq 2$, $0<s\leq\infty$ and $\frac{N+2}{2}<\kappa\leq N+2, \frac{N}{2}<\vartheta\leq N$.  \\
  Furthermore, according to  \cite[Proof of Theorem 2.21]{55QH2} we verify that for $q>1,0<\vartheta<\min\{N,q\} $ and $\varphi\in L^1(0,T,W_0^{1,1}(\Omega))$ there holds
  \begin{equation}\label{161120149}                                 
  \left(\int_0^T|\text{osc}_{B_\rho\cap\overline{\Omega}}\varphi(t)|^qdt\right)^{\frac{1}{q}} \leq C \rho^{1-\frac{\vartheta}{q}} |||\nabla \varphi|||_{L_{**}^{q,q;\vartheta}(\Omega_T)},
  \end{equation}
  for any ball $B_\rho\subset\mathbb{R}^N$, where $C=C(N,q,\vartheta)$. Therefore, \eqref{191120141} implies a global Holder-estimate in space variable and $L^q-$estimate  in time, namely for all ball $B_\rho\subset\mathbb{R}^N$
  \begin{equation*}                                
  \left(\int_0^T|\text{osc}_{B_\rho\cap\overline{\Omega}}u(t)|^qdt\right)^{\frac{1}{q}} \leq C \rho^{1-\frac{\vartheta}{q}} \left( ||\mathcal{M}_1[\omega]||_{L_{**}^{q,q;\vartheta}(\Omega_T)}+ |||F|||_{L_{**}^{q,q;\vartheta}(\Omega_T)}\right),
  \end{equation*}
  with $0<\vartheta<\min\{N,q\}.$     

We would like to refer to  \cite{Mi2,55Mi0} as the first papers which have been used the first order factional maximal operators in order to obtain the  Lorentz-Morrey
estimates for gradients of solutions to nonlinear elliptic equations with measure or $L^1$ data.

Finally, we use Theorem \ref{161120141} to prove the  existence of solutions of the Riccati type parabolic equations
    \begin{equation}\label{111120143}
                      \left\{
                      \begin{array}
                      [c]{l}%
                      {u_{t}}-\operatorname{div}(A(x,t,\nabla u))=|\nabla u|^q+\operatorname{div}(F)+\mu~~~\text{in }\Omega_T,\\
                      {u}=0\qquad\text{on }\partial\Omega\times (0,T),\\
                      u(0)=\sigma ~~\text{ in }~ \Omega,
                      \end{array}
                      \right.  
                      \end{equation}                        
 where $q>1$ and  $F\in L^q(\Omega_T,\mathbb{R}^N)$, $\mu\in\mathfrak{M}_b(\Omega_T)$, $\sigma\in\mathfrak{M}_b(\Omega)$. \\
    \begin{theorem}\label{5hh260320144} Suppose that $A$ is linear. Let $q>1$, $F\in L^q(\Omega_T,\mathbb{R}^N)$ and $\mu\in\mathfrak{M}_b(\Omega_T), \sigma\in\mathfrak{M}_b(\Omega)$,  set $\omega=|\mu|+|\sigma|\otimes\delta_{\{t=0\}}$. There exist $C_1=C_1(N,\Lambda_1,\Lambda_2,q,T_0)$, $\delta=\delta(N,\Lambda_1,\Lambda_2,q)\in (0,1)$ and  $s_0=s_0(N,\Lambda_1,\Lambda_2)>0$ such that if $\Omega$ is a $(\delta,R_0)$-Reifenberg flat domain and $[A]_{s_0}^{R_0}\le \delta$ for some $R_0>0$ and 
                            \begin{align}\label{171120141}
                            \omega(K)\leq C_1 \operatorname{Cap}_{\mathcal{G}_1,q'}(K),
                            \end{align}
                            and 
                            \begin{align}\label{171120142}
                           \int_{K}H_qdxdt\leq  C_1^q \operatorname{Cap}_{\mathcal{G}_1,q'}(K),
                            \end{align}    for any compact set $K\subset \mathbb{R}^N$ where $H_q=\left(\mathcal{M}(|F|^2)\right)^{q/2}\chi_{\Omega_T}$ if $q\geq \frac{N+2}{N}$ and $H_q=|F|^q\chi_{\Omega_T}$ if $q< \frac{N+2}{N}$ , then problem \eqref{111120143} 
                         has a  weak solution $u\in L^q(0,T, W_0^{1,q}(\Omega))$ satisfying 
                                                  \begin{align*}
                                        \int_{K\cap \Omega_T}|\nabla u|^qdxdt\leq  C_2 \operatorname{Cap}_{\mathcal{G}_1,q'}(K),     
                                                  \end{align*}
                                                  for any compact set $K\subset \mathbb{R}^N$, here $C_2=C_2(N,\Lambda_1,\Lambda_2,q,T_0/R_0,T_0, C_1)>0$.   
                                                   \end{theorem}                    
In this Theorem, capacity $\operatorname{Cap}_{\mathcal{G}_1,q'}$ denotes the  $(\mathcal{G}_1,q')$-capacity where $\mathcal{G}_1$ is the  Bessel parabolic kernel of first order (see \cite{55Bag})
\begin{equation*}
     \mathcal{G}_1(x,t)=\left((4\pi)^{N/2}\Gamma(1/2)\right)^{-1}                            \frac{\chi_{(0,\infty)}(t)}{t^{(N+1)/2}}\exp\left(-t-\frac{|x|^2}{4t}\right)~~\text{for}~~(x,t)~~\text{in}~~\mathbb{R}^{N+1}.
                                     \end{equation*}  
It is defined by 
    \begin{align*}               
   \operatorname{Cap}_{\mathcal{G}_1,q'}(E)=\inf\left\{\int_{\mathbb{R}^{N+1}}|f|^{q'}dxdt: f\in L^{q'}_+(\mathbb{R}^{N+1}), \mathcal{G}_1*f\geq \chi_E\right\},
                \end{align*} 
   for any Borel set $E\subset\mathbb{R}^{N+1}$, where $\chi_E$ is the characteristic function on $E$.                
      Note that if $1<q<\frac{N+2}{N+1}$,  the capacity $ \operatorname{Cap}_{\mathcal{G}_1,q'}$ of a singleton is positive thus \eqref{171120141} and \eqref{171120142} hold for some constant $C_1>0$ provided $\mu\in\mathfrak{M}_b(\Omega_T), u_0\in\mathfrak{M}_b(\Omega)$ and $|F|\in L^q(\Omega_T)$.\\
      We remark that in case $F\equiv 0$ the existence of solutions to \eqref{111120143} has been obtained in our paper \cite{55QH2}.
      \begin{remark} The  inequality \eqref{171120141} is equivalent to 
      \begin{align}\label{171120143}
                            & |\mu|(K)\leq C \operatorname{Cap}_{\mathcal{G}_1,q'}(K),~~\sigma\equiv 0 \text{ when }~q\geq 2,\\&
  |\mu|(K)\leq C \operatorname{Cap}_{\mathcal{G}_1,q'}(K),~~|\sigma|(O)\leq \operatorname{Cap}_{\mathbf{G}_{\frac{2-q}{q}},q'}(O) \text{ when }~1<q< 2,                     
                             \end{align}
                             for any compact sets $K\subset \mathbb{R}^{N+1}, O\subset\mathbb{R}^N,$ where $\mathbf{G}_{\frac{2-q}{q}}$ is  the  Bessel  kernel of order $\frac{2-q}{q}$ and  capacity $\operatorname{Cap}_{\mathbf{G}_{\frac{2-q}{q}},q'}$ of $O$ is defined by                             
                                 \begin{align*}               
                              \operatorname{Cap}_{\mathbf{G}_{\frac{2-q}{q}},q'}(O)=\inf\left\{\int_{\mathbb{R}^{N}}|f|^{q'}dx: f\in L^{q'}_+(\mathbb{R}^{N}), \mathbf{G}_{\frac{2-q}{q}}*f\geq \chi_O\right\},
                                             \end{align*}
    see \cite[Remark 4.34]{55QH2}.   Moreover, if $q>2$,   the inequality \eqref{171120142} is equivalent to \begin{align*}
                           \int_{K\cap \Omega_T} |F|^q dxdt\leq  C \operatorname{Cap}_{\mathcal{G}_1,q'}(K),
                                                     \end{align*} 
                               for any compact set $K\subset \mathbb{R}^N$, see Lemma \ref{171120146}. 
      \end{remark}

    \section{Interior estimates and boundary estimates for parabolic equations}
    In this section we present various local interior and boundary estimates for weak solution $u$ of \eqref{5hh070120148}. They will be used for our global estimates later. First we recall basic existence and uniqueness result of problem \eqref{5hh070120148}.
    \begin{proposition}\label{111120144} If  $F\in L^2(\Omega_T,\mathbb{R}^N)$, there exists a unique weak solution $u\in L^2(0,T;H^1_0(\Omega))$ of \eqref{5hh070120148} and the following global estimate holds: 
    \begin{align}\label{111120145}
    \int_{\Omega_T}|\nabla u|^2dxdt\leq \Lambda_2^{-1/2} \int_{\Omega_T}|F|^2dxdt.
    \end{align}
    \end{proposition}
    The existence and uniqueness of a weak solution of problem \eqref{5hh070120148} with $F\in L^2(\Omega_T,\mathbb{R}^N)$ is obtained from the Lax-Milgram Theorem, version for parabolic  framework.
    Using $u$ as a test function in \eqref{5hh070120148}, we get \eqref{111120145}. Moreover, due to the embedding 
    \begin{align*}
    \{\varphi:\varphi\in L^2(0,T;H_0^1(\Omega)),\varphi_t\in L^2(0,T;H^{-1}(\Omega))\}\subset C(0,T;L^2(\Omega)),
    \end{align*} 
    thus, the unique weak solution $u$ of \eqref{5hh070120148} belongs to $C(0,T;L^2(\Omega)).$ We can see that $u$ is also the unique weak solution of \eqref{5hh070120148} in $\Omega\times (-\infty,T)$ where $F\in L^2(\Omega_T,\mathbb{R}^N)$ and $F=0, u=0$ in $\Omega\times (-\infty,0)$.\medskip\\
   For some technical reasons, throughout this section, we always assume that  $u\in C(-\infty, T;L^2(\Omega))\cap L^2(-\infty,T;H^1_0(\Omega))$ is a weak solution to equation \eqref{5hh070120148} in $\Omega\times (-\infty,T)$ with  $F\in L^2(\Omega_T,\mathbb{R}^N)$, $F=0$ in $\Omega\times (-\infty,0)$. \\
      \subsection{Interior Estimates}
       For each ball $B_{2R}=B_{2R}(x_0)\subset\subset\Omega$ and $t_0\in (0,T)$, one considers the unique solution 
      \begin{equation*}
      w\in C(t_0-4R^2,t_0;L^2(B_{2R}))\cap L^2(t_0-4R^2,t_0;H^1(B_{2R}))
      \end{equation*}
      to the following equation 
      \begin{equation}
       \label{111120146}\left\{ \begin{array}{l}
         {w_t} - \operatorname{div}\left( {A(x,t,\nabla w)} \right) = 0 \;in\;Q_{2R}, \\ 
         w = u\quad \quad on~~\partial_{p}Q_{2R}, \\ 
         \end{array} \right.
       \end{equation}
       where $Q_{2R}=B_{2R}  \times (t_0-4R^2,t_0)$ and $\partial_{p}Q_{2R}= \left( {\partial B_{2R}  \times (t_0-4R^2,t_0)} \right) \cup \left( {B_{2R}  \times \left\{ {t = t_0-4R^2} \right\}} \right) $.\\
       The following a variant of Gehring's lemma was proved in \cite{55Nau,55DuzaMing}. 
       \begin{lemma} \label{111120147} Let $w$ be in \eqref{111120146}.
      There exist  constants $\theta_1>2$ and $C$ depending only on $N,\Lambda_1,\Lambda_2$ such that the following estimate      
       \begin{equation}\label{111120148}
       \left(\fint_{Q_{\rho/2}(y,s)}|\nabla w|^{\theta_1} dxdt\right)^{\frac{1}{\theta_1}}\leq C\fint_{Q_{\rho}(y,s)}|\nabla w| dxdt,
       \end{equation}holds 
               for all  $Q_{\rho}(y,s)\subset Q_{2R}$. 
       \end{lemma} 
       The next lemma gives an estimate for $\nabla u-\nabla w$.
       \begin{lemma}\label{111120149}Let $w$ be in \eqref{111120146}. There exists a constant $C=C(N,\Lambda_1,\Lambda_2)>0$ such that
       \begin{align}\label{1111201410}
       \fint_{Q_{2R}}|\nabla u-\nabla w|^2dxdt\leq C\fint_{Q_{2R}}|F|^2dxdt.
       \end{align}
       \end{lemma}
       \begin{proof} Using $u-w$ as a test function in \eqref{5hh070120148} and \eqref{111120146} and since \begin{align*}
                    \int_{Q_{2R}}u_t(u-w)dxdt- \int_{Q_{2R}}w_t(u-w)dxdt =\frac{1}{2}\int_{B_{2R}}(u-w)^2(t_0)dx\geq  0,
                    \end{align*}
     we find
     \begin{align*}
     \int_{Q_{2R}} \left \langle A(x,t,\nabla u)-A(x,t,\nabla w),\nabla u-\nabla w\right\rangle dxdt\leq \int_{Q_{2R}} \left \langle F,\nabla u-\nabla w\right\rangle dxdt. 
     \end{align*} 
     Using  \eqref{5hhcondb} and H\"older inequality we derive \eqref{1111201410}. 
       \end{proof}\medskip\\\\
      To continue, we denote by $v$ the unique function
       \begin{equation*}
        v\in C(t_0-R^2,t_0;L^2(B_{R}))\cap L^2(t_0-R^2,t_0;H^1(B_{R}))
        \end{equation*}
        solution of the following equation 
         \begin{equation}\label{5hheq4}
         \left\{ \begin{array}{l}
              {v_t} - \operatorname{div}\left( {\overline{A}_{B_R(x_0)}(t,\nabla v)} \right) = 0 \;in\;Q_{R}, \\ 
              v = w\quad \quad on~~\partial_{p}Q_{R}, \\ 
              \end{array} \right.
         \end{equation}
          where $Q_{R}=B_{R}(x_0)  \times (t_0-R^2,t_0)$ and $\partial_{p}Q_{R}= \left( {\partial B_{R}  \times (t_0-R^2,t_0)} \right) \cup \left( {B_{R}  \times \left\{ {t = t_0-R^2} \right\}} \right) $.
       \begin{lemma}\label{5hh21101319} Let $\theta_1$ be  the constant in Lemma \ref{111120147}. There exist  constants $C_1=C_1(N,\Lambda_1,\Lambda_2)$ and  $C_2=C_2(\Lambda_1,\Lambda_2)$ such that  \begin{eqnarray}
               \fint_{Q_R}|\nabla w-\nabla v|^2dxdt\leq C_1 \left([A]_{s_1}^{R}\right)^2 \fint_{Q_{2R}}|\nabla w|^2dxdt, \label{5hh18094}
               \end{eqnarray}
       with $s_1=\frac{2\theta_1}{\theta_1-2}$ and \begin{equation}\label{5hh18091}
               C_2^{-1} \int_{Q_R}|\nabla v|^2dxdt\leq  \int_{Q_R}|\nabla w|^2dxdt\leq  C_2\int_{Q_R}|\nabla v|^2dxdt.
                \end{equation}
       \end{lemma}
       \begin{proof} The proof can be found in \cite[Lemma 7.3]{55QH2}. 
       \end{proof}
       \begin{theorem}\label{5hh24092}Let $\theta_1$ be  the constant in Lemma \ref{111120147}. 
       There exists a functions $v\in C(t_0-R^2,t_0;L^2(B_{R}))\cap L^2(t_0-R^2,t_0;H^1(B_{R}))\cap L^\infty(t_0-\frac{1}{4}R^2,t_0;W^{1,\infty}(B_{R/2}))$ such that 
       \begin{equation}\label{5hh18092}
       ||\nabla v||^2_{L^\infty(Q_{R/2})}\leq C \fint_{Q_{2R}}|\nabla u|^2 dxdt +C\fint_{Q_{2R}}|F|^2 dxdt,
       \end{equation}
       and 
         \begin{eqnarray}
               \fint_{Q_R}|\nabla u-\nabla v|^2dxdt\leq  C\fint_{Q_{2R}}|F|^2 dxdt+ C\left([A]_{s_1}^{R}\right)^2\left(\fint_{Q_{2R}}|\nabla u|^2dxdt+ \fint_{Q_{2R}}|F|^2 dxdt\right),\label{5hh18093}
               \end{eqnarray}
               where $s_1=\frac{2\theta_1}{\theta_1-2}$ and $C=C(N,\Lambda_1,\Lambda_2)$.
       \end{theorem}
       \begin{proof}
       Let $w$ and $v$ be in equations \eqref{111120146} and \eqref{5hheq4}. By standard interior regularity and inequality \eqref{111120148} in Lemma \ref{111120147} and \eqref{5hh18091} in Lemma \ref{5hh21101319} we have 
       \begin{align*}
      ||\nabla v||_{L^\infty(Q_{R/2})}&\leq C \left(\fint_{Q_R}|\nabla v|^2dxdt\right)^{1/2}\\&\leq C \left(\fint_{Q_R}|\nabla w|^2dxdt\right)^{1/2}.
       \end{align*}   
       Thus, we get \eqref{5hh18092}
        from inequality \eqref{1111201410}  in Lemma \ref{111120149}.\\
        On the other hand, applying \eqref{5hh18094} in Lemma \ref{5hh21101319}  yields        
             \begin{eqnarray*}
                        \fint_{Q_R}|\nabla u-\nabla v|^2dxdt\leq \fint_{Q_R}|\nabla u-\nabla w|^2dxdt+ c_4 \left([A]_{s_1}^{R}\right)^2\fint_{Q_{2R}}|\nabla w|^2 dxdt. 
                         \end{eqnarray*}
                Hence, we get \eqref{5hh18093}
                            from \eqref{1111201410}  in Lemma \ref{111120149}. The proof is complete. 
       \end{proof}  
       \subsection{Boundary Estimates}
       In this subsection, we focus on the corresponding estimates near the boundary. \\
      Throughout this subsection, we always assume that $\Omega$ is a $(\delta,R_0)$- Reifenberg flat domain with $0<\delta\leq 1/2$.  In particular, we can see that the complement of $\Omega$ is uniformly 2-thick for some constants $c_0,r_0$, see \cite{55QH2}. Let $x_0\in \partial\Omega$ be a boundary point and  $0<R<R_0/6$ and $t_0\in (0,T)$,  we set $\tilde{\Omega}_{6R}=\tilde{\Omega}_{6R}(x_0,t_0)=\left(\Omega\cap B_{6R}(x_0)\right)\times (t_0-(6R)^2,t_0)$ and $Q_{6R}=Q_{6R}(x_0,t_0)$.
      
       We now consider the unique solution $w$ to the equation
       \begin{equation}
         \label{141120142}\left\{ \begin{array}{l}
           {w_t} - \operatorname{div}\left( {A(x,t,\nabla w)} \right) = 0 \;in\;\tilde{\Omega}_{6R}, \\ 
           w = u\quad \quad on~~\partial_{p}\tilde{\Omega}_{6R}. \\ 
           \end{array} \right.
         \end{equation}
         In what follows we extend $F$ and $u$ by zero to $\left(\Omega\times (-\infty,T)\right)^c$ and then extend $w$ by $u$ to $\mathbb{R}^{N+1}\backslash \tilde{\Omega}_{6R}$.
         \begin{lemma}\label{141120141}
         Let $w$ be in \eqref{141120142}.
               There exist  constants $\theta_2>2$ and $C>0$ depending only on $N,\Lambda_1,\Lambda_2$ such that the following estimate    
                \begin{equation}\label{141120143}
                \left(\fint_{Q_{\rho/2}(y,s)}|\nabla w|^{\theta_2} dxdt\right)^{\frac{1}{\theta_2}}\leq C\fint_{Q_{3\rho}(y,s)}|\nabla w| dxdt,
                \end{equation}
                 holds  for all  $Q_{3\rho}(z,s)\subset Q_{6R}$.
         \end{lemma}
         Above lemma was proved in \cite[Theorem 7.5]{55QH2}. Analogous to Lemma \ref{111120149} we obtain
          \begin{lemma}\label{141120144}Let $w$ be in \eqref{141120142}. There exists a constant $C=C(N,\Lambda_1,\Lambda_2)>0$ such that
                \begin{align}\label{141120145}
                \fint_{Q_{6R}}|\nabla u-\nabla w|^2dxdt\leq C\fint_{Q_{6R}}|F|^2dxdt.
                \end{align}
                \end{lemma} 
  Next, we set $\rho=R(1-\delta)$ so that $0<\rho/(1-\delta)<R_0/6$. By the definition of Reifenberg flat domains,  there exists a coordinate system $\{y_1,y_2,...,y_N\}$ with the
  origin $0\in\Omega$ such that in this coordinate system $x_0=(0,...,0,-\rho\delta/(1-\delta))$ and 
  \begin{equation*}
  B^+_\rho(0)\subset \Omega\cap B_\rho(0)\subset B_\rho(0)\cap \{y=(y_1,y_2,....,y_N):y_N>-2\rho\delta/(1-\delta)\}.
  \end{equation*}
  Since $\delta<1/2$ we have 
   \begin{equation}\label{5hh1610138}
    B^+_\rho(0)\subset \Omega\cap B_\rho(0)\subset B_\rho(0)\cap \{y=(y_1,y_2,....,y_N):y_N>-4\rho\delta\},
    \end{equation}
    where $B^+_\rho(0):=B_\rho(0)\cap\{y=(y_1,y_2,...,y_N):y_N>0\}$.\\
    Furthermore we consider the unique solution
           \begin{equation*}
            v\in C(t_0-\rho^2,t_0;L^2(\Omega\cap B_\rho(0)))\cap L^2(t_0-\rho^2,t_0;H^1(\Omega\cap B_\rho(0)))
            \end{equation*}
            to the following equation
             \begin{equation}\label{5hh1610131}
             \left\{ \begin{array}{l}
                  {v_t} - \operatorname{div}\left( {\overline{A}_{B_{\rho}(0)}(t,\nabla v)} \right) = 0 \;in\;\tilde{\Omega}_\rho(0),\\ 
                  v = w\quad \quad on~~\partial_{p}\tilde{\Omega}_\rho(0), \\ 
                  \end{array} \right.
             \end{equation}
              where $\tilde{\Omega}_\rho(0)=\left(\Omega\cap B_{\rho}(0)\right)\times (t_0-\rho^2,t_0)$ ($0<t_0<T$). \\
              We put $v=w$ outside $\tilde{\Omega}_\rho(0)$. As Lemma \ref{5hh21101319} we have the following result. 
               \begin{lemma}\label{5hh1610139} Let $\theta_2$ be  the constant in Lemma \ref{141120141}. There exist positive constants $C_1=C_1(N,\Lambda_1,\Lambda_2)$ and  $C_2=C_2(\Lambda_1,\Lambda_2)$ such that  \begin{eqnarray}
                             \fint_{Q_{\rho}(0,t_0)}|\nabla w-\nabla v|^2dxdt\leq C_1 \left([A]_{s_2}^{R}\right)^2 \fint_{Q_{\rho}(0,t_0)}|\nabla w|^2dxdt, \label{5hh1610132}
                             \end{eqnarray}
                     with $s_1=\frac{2\theta_2}{\theta_2-2}$ and \begin{equation}\label{5hh1610133}
                             C_2^{-1} \int_{Q_{\rho}(0,t_0)}|\nabla v|^2dxdt\leq  \int_{Q_{\rho}(0,t_0)}|\nabla w|^2dxdt\leq  C_2\int_{Q_{\rho}(0,t_0)}|\nabla v|^2dxdt.
                              \end{equation}
                     \end{lemma}              
                    We can see that if the boundary of $\Omega$ is irregular enough, then the $L^\infty$-norm of $\nabla v$ up to $\partial\Omega\cap B_\rho(0)\times (t_0-\rho^2,t_0)$ may not exist. For our purpose, we will consider another  equation: 
            \begin{equation}\label{5hh1610134}
               \left\{ \begin{array}{l}
       {V_t} - \operatorname{div}\left( {\overline{A}_{B_{\rho}(0)}(t,\nabla V)} \right) = 0 ~~\text{in}~~Q_\rho^+(0,t_0),\\ 
                V = 0\quad \quad \text{on}~~ T_\rho(0,t_0), \\ 
                            \end{array} \right.
                                             \end{equation}
                                             where $Q_\rho^+(0,t_0)=B_\rho^+(0)\times (t_0-\rho^2,t_0)$ and  $T_\rho(0,t_0)=Q_\rho(0,t_0)\cap\{x_N=0\}$.\\
A weak solution $V$ of above problem is understood in the following sense: the zero extension of $V$  to $Q_\rho(0,t_0)$ is in  $ C(t_0-\rho^2,t_0;L^2( B_\rho(0)))\cap L^2_{\text{loc}}(t_0-\rho^2,t_0;H^1( B_\rho(0)))$  and for every $\varphi\in C_c^1(Q_\rho^+(0,t_0))$ there holds
\begin{align*}
-\int_{Q_\rho^+(0,t_0)}V\varphi_tdxdt+\int_{Q_\rho^+(0,t_0)}\overline{A}_{B_{\rho}(0)}(t,\nabla V)\nabla \varphi dxdt=0. 
\end{align*}                                      
                                         We have the following  $L^\infty$ gradient estimate up to the boundary for $V$.   The following Lemma was obtained in \cite[Lemma 7.12]{55QH2}.

    \begin{lemma}\label{5hh21101314}
        For any $\varepsilon>0$ there exists a small $\delta_0=\delta_0(N,\Lambda_1,\Lambda_2,\varepsilon)\in (0,1/2)$ such that if $v\in C(t_0-\rho^2,t_0;L^2(\Omega\cap B_\rho(0)))\cap L^2(t_0-\rho^2,t_0;H^1(\Omega\cap B_\rho(0)))$ is a solution of \eqref{5hh1610131}
       and under condition \eqref{5hh1610138} with $\delta\in (0,\delta_0)$, there exists a weak solution $V\in C(t_0-\rho^2,t_0;L^2( B_\rho^+(0)))\cap L^2(t_0-\rho^2,t_0;H^1( B_\rho^+(0)))$ of  \eqref{5hh1610134}, whose zero extension to $Q_\rho(0,t_0)$ satisfies 
 \begin{align*}
 &||\nabla V||^2_{L^\infty(Q_{\rho/4}(0,t_0))}\leq C \fint_{Q_{\rho}(0,t_0)} |\nabla v|^2dxdt, \\&
  \fint_{Q_{\rho/8}(0,t_0)}|\nabla v-\nabla V|^2dxdt\leq \varepsilon^2\fint_{Q_{\rho}(0,t_0)} |\nabla v|^2dxdt,
 \end{align*}
    for some $C=C(N,\Lambda_1,\Lambda_2)>0$.
        \end{lemma}       
       
\begin{theorem}\label{5hh16101310}
 Let $s_2$ be as in Lemma \ref{5hh1610139}. For any $\varepsilon>0$ there exists a small $\delta_0=\delta_0(N,\Lambda_1,\Lambda_2,\varepsilon)\in (0,1/2)$ such that the following holds. If $\Omega$ is a $(\delta,R_0)$-Reifenberg flat domain with $\delta\in (0,\delta_0)$, there is a function $V\in  L^2(t_0-(R/9)^2,t_0;H^1( B_{R/9}(x_0)))\cap L^\infty(t_0-(R/9)^2,t_0;W^{1,\infty}( B_{R/9}(x_0)))$ such that 
 \begin{equation}\label{5hh21101317}
 ||\nabla V||^2_{L^\infty(Q_{R/9}(x_0,t_0))}\leq C\fint_{Q_{6R}(x_0,t_0)}|\nabla u|^2dxdt+C \fint_{Q_{6R}(x_0,t_0)}|F|^2dxdt,
 \end{equation}
 and 
 \begin{align}
 \nonumber&\fint_{Q_{R/9}(x_0,t_0)}|\nabla u-\nabla V|^2dxdt\\&~~~~\leq C (\varepsilon^2+([A]_{s_2}^{R_0})^2)\fint_{Q_{6R}(x_0,t_0)}|\nabla u|^2dxdt+ C(\varepsilon^2+1+([A]_{s_2}^{R_0})^2)\fint_{Q_{6R}(x_0,t_0)}|F|^2dxdt,\label{5hh21101318}
 \end{align}
 for some $C=C(N,\Lambda_1,\Lambda_2)>0$.
 \end{theorem}
    \begin{proof}
 Let $x_0\in \partial \Omega$, $0<t_0<T$ and $\rho=R(1-\delta)$, we may assume that $0\in \Omega$, $x_0=(0,...,-\delta\rho/(1-\delta))$
 and 
    \begin{equation}
                                     B^+_\rho(0)\subset \Omega\cap B_\rho(0)\subset B_\rho(0)\cap \{x_N>-4\rho\delta\}.
                                                              \end{equation}
We have also 
\begin{equation}\label{5hh090520141}
Q_{R/9}(x_0,t_0)\subset Q_{\rho/8}(0,t_0)\subset Q_{\rho/4}(0,t_0)\subset Q_{\rho}(0,t_0)\subset Q_{6\rho}(0,t_0)\subset Q_{6R}(x_0,t_0),
\end{equation}                                                         provided that $0<\delta<1/625$.\\
Let $w$ and $v$ be as in Lemma \ref{141120144} and Lemma  \ref{5hh1610139}. 
    By Lemma  \ref{5hh21101314} for any $\varepsilon>0$                                                          we can find a small positive $\delta=\delta(N,\Lambda_1,\Lambda_2,\varepsilon)<1/625$ such that there is a function $V\in  L^2(t_0-\rho^2,t_0;H^1( B_{\rho}(0)))\cap L^\infty(t_0-\rho^2,t_0;W^{1,\infty}( B_{\rho}(0)))$ satisfying 
    \begin{align*}
   &  ||\nabla V||^2_{L^\infty(Q_{\rho/4}(0,t_0))}\leq c_1 \fint_{Q_{\rho}(0,t_0)} |\nabla v|^2dxdt,\\&
   \fint_{Q_{\rho/8}(0,t_0)}|\nabla v-\nabla V|^2\leq \varepsilon^2\fint_{Q_{\rho}(0,t_0)} |\nabla v|^2dxdt.
    \end{align*}    
        Then, by \eqref{5hh1610133} in Lemma  \ref{5hh1610139} and \eqref{5hh090520141} we get 
        \begin{align}
        \nonumber
                 ||\nabla V||^2_{L^\infty(Q_{R/9}(x_0,t_0))} &\leq c_2 \fint_{Q_{\rho}(0,t_0)} |\nabla w|^2dxdt\nonumber\\&\leq c_3 \fint_{Q_{6R}(x_0,t_0)} |\nabla w|^2dxdt, \label{5hh21101316}
        \end{align}       
        and 
        \begin{align}
                \fint_{Q_{R/9}(x_0,t_0)}|\nabla v-\nabla V|^2dxdt\leq c_4 \varepsilon^2\fint_{Q_{6R}(x_0,t_0)} |\nabla w|^2dxdt.\label{5hh21101320}
        \end{align}
        Therefore, from \eqref{141120145} in Lemma  \ref{141120144} and \eqref{5hh21101316} we get \eqref{5hh21101317}.\\
         Next we prove \eqref{5hh21101318}. Since \eqref{5hh090520141}, we have
         \begin{align*}
         &\fint_{Q_{R/9}(x_0,t_0)}|\nabla u-\nabla V|^2dxdt\leq c_5\fint_{Q_{\rho/8}(0,t_0)}|\nabla u-\nabla V|^2dxdt
                    \\&~~~~~~~~\leq c_6\fint_{Q_{\rho/8}(0,t_0)}|\nabla u-\nabla w|^2dxdt+c_6\fint_{Q_{\rho/8}(0,t_0)}|\nabla w-\nabla v|^2dxdt\\&~~~~~~~~~~+c_6\fint_{Q_{\rho/8}(0,t_0)}|\nabla v-\nabla V|^2dxdt.
         \end{align*}
   Using \eqref{141120145} in Lemma \ref{141120144} and \eqref{5hh1610132}, \eqref{5hh1610133} in Lemma \ref{5hh1610139} and \eqref{5hh21101320} we find that \begin{align*}
           &\fint_{Q_{\rho/8}(0,t_0)}|\nabla u-\nabla w|^2dxdt\leq c_6 \fint_{Q_{6R}(x_0,t_0)}|F|^2dxdt,
           \\&\fint_{Q_{\rho/8}(0,t_0)}|\nabla v-\nabla w|^2dxdt \leq c_{7}([A]_{s_2}^{R_0})^2 \fint_{Q_{6R}(0,t_0)}|\nabla w|^2dxdt\\&~~~~~~~~~~~~~~~~~~~~~~~~~~~~~~~~~~\leq  c_{8}([A]_{s_2}^{R_0})^2\left(\fint_{Q_{6 R}(x_0,t_0)}|\nabla u|^2dxdt+\fint_{Q_{6R}(x_0,t_0)}|F|^2dxdt\right),
           \end{align*}
           and 
           \begin{align*}           
          \fint_{Q_{\rho/8}(0,t_0)}|\nabla v-\nabla V|^2 dxdt \leq c_{9}\varepsilon^2 \left(\fint_{Q_{6 R}(x_0,t_0)}|\nabla u|^2dxdt+\fint_{Q_{6R}(x_0,t_0)}|F|^2dxdt\right).
           \end{align*}                                                
           Then we derive \eqref{5hh21101318}.  This completes the proof.

    \end{proof}  
    \section{Global integral gradient bounds for parabolic equations }
    The following good-$\lambda$ type estimate will be essential for our global estimates later. 
  \begin{theorem}\label{5hh23101312} Let $s_1,s_2$ be as in Lemma \ref{5hh21101319}, \ref{5hh1610139} and $s_0=\max\{s_1,s_2\}$. Let $w\in \mathbf{A}_\infty$, $F\in L^2(\Omega_T,\mathbb{R}^N)$. Let $u\in L^2(0,T;H^1_0(\Omega))$ be the weak solution to equation \eqref{5hh070120148} in $\Omega_T$. For any $\varepsilon>0,R_0>0$ one finds  $\delta_1=\delta_1(N,\Lambda_1,\Lambda_2,\varepsilon,[w]_{\mathbf{A}_\infty})\in (0,1/2)$ and $\delta_2=\delta_2(N,\Lambda_1,\Lambda_2,\varepsilon,[w]_{\mathbf{A}_\infty},T_0/R_0)\in (0,1)$ and $\Lambda=\Lambda(N,\Lambda_1,\Lambda_2)>0$ such that if $\Omega$ is  a $(\delta_1,R_0)$- Reifenberg flat domain and $[A]_{s_0}^{R_0}\le \delta_1$ then 
   \begin{equation}\label{5hh16101311}
   w(\{\mathcal{M}(|\nabla u|^2)>\Lambda\lambda, \mathcal{M}(|F|^2)\le \delta_2\lambda \}\cap \Omega_T)\le B\varepsilon w(\{ \mathcal{M}(|\nabla u|^2)> \lambda\}\cap \Omega_T)
   \end{equation}
   for all $\lambda>0$, 
      where the constant $B$  depends only on $N,\Lambda_1,\Lambda_2, T_0/R_0, [w]_{\mathbf{A}_\infty}$.
  \end{theorem}
  To prove above estimate, we will use  L. Caddarelli and I. Peral's technique in \cite{CaPe}. Namely, it is based on the following technical lemma whose proof is  a consequence of Lebesgue Differentiation Theorem and the standard Vitali covering  lemma, can be found in  \cite{55BW4,55MePh2} with some modifications to fit the setting here.

       \begin{lemma}\label{5hhvitali2} Let $\Omega$ be a $(\delta,R_0)$-Reifenberg flat domain with $\delta<1/4$ and let $w$ be an $\mathbf{A}_\infty$ weight. Suppose that the sequence of balls $\{B_r(y_i)\}_{i=1}^L$ with centers $y_i\in\overline{\Omega}$ and  radius $r\leq R_0/4$ covers $\Omega$. Set $s_i=T-ir^2/2$ for all $i=0,1,...,[\frac{2T}{r^2}]$. Let $E\subset F\subset \Omega_T$ be measurable sets for which there exists $0<\varepsilon<1$ such that  $w(E)<\varepsilon w(\tilde{Q}_r(y_i,s_j))$ for all $i=1,...,L$, $j=0,1,...,[\frac{2T}{r^2}]$; and  for all $(x,t)\in \Omega_T$, $\rho\in (0,2r]$, we have
               $\tilde{Q}_\rho(x,t)\cap \Omega_T\subset F$      
               if $w(E\cap \tilde{Q}_\rho(x,t))\geq \varepsilon w(\tilde{Q}_\rho(x,t))$. Then $
               w(E)\leq \varepsilon Bw(F)$         
               for a constant $B$ depending only on $N$ and $[w]_{\mathbf{A}_\infty}$.
              \end{lemma}   
  \begin{proof}[Proof of Theorem \ref{5hh23101312}]Note that  $[A]_{s_1}^{R_0},[A]_{s_2}^{R_0}\leq [A]_{s_0}^{R_0}$. Let $\varepsilon \in (0,1)$.  Set $E_{\lambda,\delta_2}=\{\mathcal{M}(|\nabla u|^2)>\Lambda\lambda, \mathcal{M}(|F|^2)\leq \delta_2\lambda\}\cap \Omega_T $ and $F_\lambda=\{\mathcal{M}(|\nabla u|^2)>\lambda \}\cap \Omega_T$ for $\delta_2\in (0,1),\Lambda>0$ and $\lambda>0$.
    Let $\{y_i\}_{i=1}^L\subset \Omega$ and a ball $B_0$ with radius $2T_0$ such that 
   $$
    \Omega\subset \bigcup\limits_{i = 1}^L {{B_{r_0}}({y_i})}  \subset {B_0},$$
    where $r_0=\min\{R_0/1080,T_0\}$. Let 
    $s_j=T-jr_0^2/2$ for all $j=0,1,...,[\frac{2T}{r_0^2}]$ and $Q_{2T_0}=B_0\times (T-4T_0^2,T)$. So,
    \begin{equation*}
        \Omega_T\subset \bigcup\limits_{i,j} {{Q_{r_0}}({y_i,s_j})}  \subset {Q_{2T_0}}.
        \end{equation*} 
   We verify that
   \begin{equation}\label{5hh2310131}
   w(E_{\lambda,\delta_2})\leq \varepsilon w({\tilde{Q}_{r_0}}({y_i,s_j})) ~~\forall ~\lambda>0
   \end{equation}
   for some $\delta_2$ small enough depending on $n,p,\alpha,\beta,\epsilon,[w]_{\mathbf{A}_\infty},T_0/R_0$.\\
   In fact, we can assume that $E_{\lambda,\delta_2}\not=\emptyset$ so $\int_{\Omega_T}|F|^2dxdt\leq  c_1|Q_{2T_0}|\delta_2\lambda$. Recalling that $\mathcal{M}$ is a bounded operator from $L^1(\mathbb{R}^{N+1})$ into $L^{1,\infty}(\mathbb{R}^{N+1})$, we find 
 \begin{equation*}
 |E_{\lambda,\delta_2}|\leq \frac{c_2}{\Lambda\lambda}\int_{\Omega_T}|\nabla u|^2dxdt.
 \end{equation*}
 Using \eqref{111120145} in Proposition \ref{111120144}, we get
  \begin{align*}
   |E_{\lambda,\delta_2}|&\leq \frac{c_3}{\Lambda\lambda}\int_{\Omega_T}|F|^2dxdt\\&
   \leq c_4\delta_2 |Q_{2T_0}|,
  \end{align*}
  which implies
 \begin{align*}
    w(E_{\lambda,\delta_2})\leq C\left(\frac{|E_{\lambda,\delta_2}|}{|Q_{2T_0}|}\right)^\nu w(Q_{2T_0})\leq C\left(c_4\delta_2\right)^\nu w(Q_{2T_0}),
   \end{align*}
   where $(C,\nu)=[w]_{\mathbf{A}_\infty}$. It is well-known that (see, e.g \cite{55Gra}) there exist $C_1=C_1(N,C,\nu)$ and $\nu_1=\nu_1(N,C,\nu)$ such that 
   \begin{align*}
   \frac{w(\tilde{Q}_{2T_0})}{w({\tilde{Q}_{r_0}}({y_i,s_j}))}\leq C_1\left(\frac{|\tilde{Q}_{2T_0}|}{|{\tilde{Q}_{r_0}}({y_i,s_j})|}\right)^{\nu_1}~~\forall i,j.
   \end{align*}
  Therefore,
  \begin{align*}
      w(E_{\lambda,\delta_2})\leq C\left(c_4\delta_2\right)^\nu C_1\left(\frac{|\tilde{Q}_{T_0}|}{|{\tilde{Q}_{r_0}}({y_i,s_j})|}\right)^{\nu_1} w({\tilde{Q}_{r_0}}({y_i,s_j}))
      < \varepsilon w({\tilde{Q}_{r_0}}({y_i,s_j}))~~\forall ~i,j,
     \end{align*}
     where $\delta_2\leq \varepsilon^{1/\nu}\left(2CC_1 c_4^\nu(T_0r_0^{-1})^{(N+2)\nu_1}\right)^{-1/\nu}$. Thus \eqref{5hh2310131} follows.\\
  Next we verify that for all $(x,t)\in \Omega_T$, $r\in (0,2r_0]$ and $\lambda>0$ we have
    $
     \tilde{Q}_r(x,t)\cap \Omega_T\subset F_\lambda
     $
     provided $$
        w(E_{\lambda,\delta_2}\cap \tilde{Q}_r(x,t))\geq \varepsilon w(Q_r(x,t)),
      $$
     for some $\delta_2\leq \min\left\{1,\varepsilon^{1/\nu}\left(2CC_1 c_4^\nu(T_0r_0^{-1})^{(N+2)\nu_1}\right)^{-1/\nu}\right\}$.
        Indeed,
  take $(x,t)\in \Omega_T$ and $0<r\leq 2r_0$.
             Now assume that $\tilde{Q}_r(x,t)\cap \Omega_T\cap F^c_\lambda\not= \emptyset$ and $E_{\lambda,\delta_2}\cap \tilde{Q}_r(x,t)\not = \emptyset$ i.e, there exist $(x_1,t_1),(x_2,t_2)\in \tilde{Q}_r(x,t)\cap \Omega_T$ such that $\mathcal{M}(|\nabla u|^2)(x_1,t_1)\leq \lambda$ and $\mathcal{M}(|F|^2)(x_2,t_2)\le \delta_2 \lambda$.              
              We need to prove that
              \begin{equation}\label{5hh2310133}
                     w(E_{\lambda,\delta_2}\cap \tilde{Q}_r(x,t)))< \varepsilon w(\tilde{Q}_r(x,t)). 
                                      \end{equation}
         Using $\mathcal{M}(|\nabla u|^2)(x_1,t_1)\leq \lambda$, we can see that
                                      \begin{equation*}
                                      \mathcal{M}(|\nabla u|^2)(y,s)\leq \max\left\{\mathcal{M}\left(\chi_{\tilde{Q}_{2r}(x,t)}|\nabla u|^2\right)(y,s),3^{N+2}\lambda\right\}~~\forall (y,s)\in \tilde{Q}_r(x,t).
                                      \end{equation*}
           Therefore, for all $\lambda>0$ and $\Lambda\geq 3^{N+2}$,
           \begin{eqnarray}\label{5hh2310134}E_{\lambda,\delta_2}\cap \tilde{Q}_r(x,t)=\left\{\mathcal{M}\left(\chi_{\tilde{Q}_{2r}(x,t)}|\nabla u|^2\right)>\Lambda\lambda, \mathcal{M}(|F|^2)\leq \delta_2\lambda\right\}\cap \Omega_T \cap \tilde{Q}_r(x,t).
           \end{eqnarray}
           In particular, $E_{\lambda,\delta_2}\cap \tilde{Q}_r(x,t)=\emptyset$ if $\overline{B}_{8r}(x)\subset\subset \mathbb{R}^{N}\backslash \Omega$.
           Thus, it is enough to consider the case $B_{8r}(x)\subset\subset\Omega$ and the case $B_{8r}(x)\cap\Omega\not=\emptyset$.\\   
           First assume  $B_{8r}(x)\subset\subset\Omega$. Let $v$ be as in Theorem \ref{5hh24092} with $Q_{2R}=Q_{8r}(x,t_0)$ and  $t_0=\min\{t+2r^2,T\}$. We have  
           \begin{equation}\label{5hh2310135}
                  ||\nabla v||^2_{L^\infty(Q_{2r}(x,t_0))}\leq c_5 \fint_{Q_{8r}(x,t_0)}|\nabla u|^2 dxdt +c_5\fint_{Q_{8r}(x,t_0)}|F|^2 dxdt,
                  \end{equation}
                  and 
                  \begin{align*}
              \fint_{Q_{4r}(x,t_0)}|\nabla u-\nabla v|^2dxdt&\leq  c_5\fint_{Q_{8r}(x,t_0)}|F|^2 dxdt\\&~~~~~+ c_5([A]_{s_1}^{R})^2\left(\fint_{Q_{8r}(x,t_0)}|\nabla u|^2dxdt+ \fint_{Q_{8r}(x,t_0)}|F|^2 dxdt\right).
                  \end{align*}
  Thanks to $\mathcal{M}(|\nabla u|^2)(x_1,t_1)\leq \lambda$ and $\mathcal{M}(|F|^2)(x_2,t_2)\le \delta_2 \lambda$ with $(x_1,t_1),(x_2,t_2)\in Q_r(x,t)$, we find $Q_{8r}(x,t_0)\subset\tilde{Q}_{17r}(x_1,t_1),\tilde{Q}_{17r}(x_2,t_2) $ and 
  \begin{align}\nonumber
  ||\nabla v||^2_{L^\infty(Q_{2r}(x,t_0))}&\leq c_6 \fint_{\tilde{Q}_{17r}(x_1,t_1)}|\nabla u|^2 dxdt +c_6\fint_{\tilde{Q}_{17r}(x_2,t_2)}|F|^2 dxdt\\&\nonumber\leq
  c_6(1+\delta_2)\lambda\\&\leq
  c_7\lambda,\label{1411201410}
  \end{align}
and 
\begin{align}\nonumber
\fint_{Q_{4r}(x,t_0)}|\nabla u-\nabla v|^2dxdt&\leq  c_8\delta_2\lambda+ c_5([A]_{s_0}^{R})^2(1+\delta_2)\lambda\\&\leq 
c_9(\delta_2+\delta_1^2(1+\delta_2))\lambda.\label{1411201411}
\end{align}                                                   Here we used $[A]_{s_0}^{R_0}\leq \delta_1$ in the last inequality. \\                        
   In view of \eqref{1411201410} we see that for $\Lambda\geq \max\{3^{N+2},4c_{7}\}$,
 \begin{align*}                           |\{\mathcal{M}\left(\chi_{\tilde{Q}_{2r}(x,t)}|\nabla v|^2\right)>\Lambda\lambda/4\}\cap \tilde{Q}_r(x,t)|=0.
                           \end{align*}
                           Leads to
\begin{align*}
|E_{\lambda,\delta_2}\cap \tilde{Q}_r(x,t)|&\leq   |\{\mathcal{M}\left(\chi_{\tilde{Q}_{2r}(x,t)}|\nabla u-\nabla v|^2\right)>\Lambda\lambda/4\}\cap \tilde{Q}_r(x,t)|.                                                 
\end{align*}                           
Therefore, by  bound of operator $\mathcal{M}$ from $L^1(\mathbb{R}^{N+1})$ to $L^{1,\infty}(\mathbb{R}^{N+1})$ and \eqref{1411201411}, $\tilde{Q}_{2r}(x,t)\subset Q_{4r}(x,t_0)$ we deduce
\begin{align*}
|E_{\lambda,\delta_2}\cap \tilde{Q}_r(x,t)|&\leq  \frac{c_{10}}{\lambda}\int_{\tilde{Q}_{2r}(x,t)} |\nabla u-\nabla v|^2dxdt  \\&\leq  c_{11} \left(\delta_2+\delta_1^2(1+\delta_2)\right)|Q_r(x,t)|. 
\end{align*}
Thus,  
\begin{align*}
w(E_{\lambda,\delta_2}\cap \tilde{Q}_r(x,t))&\leq C\left(\frac{|E_{\lambda,\delta_2}\cap \tilde{Q}_r(x,t) |}{|\tilde{Q}_r(x,t)|}\right)^\nu w(\tilde{Q}_r(x,t))
    \\&\leq  C\left(c_{11} \left(\delta_2+\delta_1^2(1+\delta_2)\right)\right)^\nu w(\tilde{Q}_r(x,t))
    \\&< \varepsilon w(\tilde{Q}_r(x,t)).
\end{align*} 
    where $\delta_2,\delta_1$ are appropriately chosen and  $(C,\nu)=[w]_{\mathbf{A}_\infty}$.\\
    Next assume $B_{8r}(x)\cap\Omega\not=\emptyset$. Let $x_3\in\partial \Omega$ such that $|x_3-x|=\text{dist}(x,\partial\Omega)$. Set $t_0=\min\{t+2r^2,T\}$. We have 
    \begin{equation}\label{5hh2310138}
    Q_{2r}(x,t_0)\subset Q_{10r}(x_3,t_0)\subset Q_{540r}(x_3,t_0)\subset \tilde{Q}_{1080r}(x_3,t)\subset \tilde{Q}_{1088r}(x,t)\subset \tilde{Q}_{1089r}(x_1,t_1),
    \end{equation}
    and 
    \begin{equation}\label{5hh2310139}
        Q_{540r}(x_3,t_0)\subset \tilde{Q}_{1080r}(x_3,t)\subset \tilde{Q}_{1088r}(x,t)\subset \tilde{Q}_{1089r}(x_2,t_2).
        \end{equation}
     Let $V$ be as in Theorem 
    \ref{5hh16101310} with $Q_{6R}=Q_{540r}(x_3,t_0)$ and $\varepsilon=\delta_3\in (0,1)$. We have
    \begin{equation*}
     ||\nabla V||^2_{L^\infty(Q_{10r}(x_3,t_0))}\leq c_{12}\fint_{Q_{540r}(x_3,t_0)}|\nabla u|^2dxdt+c_{12} \fint_{Q_{540r}(x_3,t_0)}|F|^2dxdt,
     \end{equation*}
     and 
     \begin{align*}
     \nonumber&\fint_{Q_{10r}(x_3,t_0)}|\nabla u-\nabla V|^2dxdt\\&~~~~\leq c_{12} (\delta_3^2+([A]_{s_2}^{R_0})^2)\fint_{Q_{540r}(x_3,t_0)}|\nabla u|^2dxdt+ c_{12}(\delta_3^2+1+([A]_{s_2}^{R_0})^2)\fint_{Q_{540r}(x_3,t_0)}|F|^2dxdt.
     \end{align*}      
Since $\mathcal{M}(|\nabla u|^2)(x_1,t_1)\leq \lambda$,  $\mathcal{M}(|F|^2)(x_2,t_2)\le \delta_2 \lambda$ and \eqref{5hh2310138}, \eqref{5hh2310139} we get 
\begin{align*}
 ||\nabla V||^2_{L^\infty(Q_{10r}(x_3,t_0))}&\leq c_{13}\fint_{\tilde{Q}_{1089r}(x_1,t_1)}|\nabla u|^2dxdt+c_{13} \fint_{\tilde{Q}_{1089r}(x_1,t_1)}|F|^2dxdt\\&\leq
 c_{14}(1+\delta_2)\lambda
 \\&\leq c_{15}\lambda,
\end{align*}
and 
\begin{align}
     \nonumber&\fint_{Q_{10r}(x_3,t_0)}|\nabla u-\nabla V|^2dxdt\leq c_{16} \left((\delta_3^2+([A]_{s_2}^{R_0})^2)+ (\delta_3^2+1+([A]_{s_2}^{R_0})^2)\delta_2\right)\lambda\\&\leq c_{16} \left((\delta_3^2+\delta_1^2)+ (\delta_3^2+1+\delta_1^2)\delta_2\right)\lambda.\label{1411201412}
     \end{align} 
  Notice that we have used  $[A]_{s_0}^{R_0}\leq \delta_1$ in the last inequality.\\  
  Now set $\Lambda= \max\{3^{N+2},4c_{7},4c_{15}\}$. As above we also have  
 \begin{align*}
 |E_{\lambda,\delta_2}\cap \tilde{Q}_r(x,t)|&\leq   |\{\mathcal{M}\left(\chi_{\tilde{Q}_{2r}(x,t)}|\nabla u-\nabla V|^2\right)>\Lambda\lambda/4\}\cap \tilde{Q}_r(x,t)|.                         
 \end{align*}                         
 Therefore using \eqref{1411201412} we obtain 
 \begin{align*}
 |E_{\lambda,\delta_2}\cap \tilde{Q}_r(x,t)|&\leq  \frac{c_{17}}{\lambda}\int_{\tilde{Q}_{2r}(x,t)} |\nabla u-\nabla V|^2dxdt \\&\leq  c_{18} \left((\delta_3^2+\delta_1^2)+ (\delta_3^2+1+\delta_1^2)\delta_2\right)|\tilde{Q}_r(x,t)|. 
 \end{align*}
 Thus
 \begin{align*}
  w(E_{\lambda,\delta_2}\cap \tilde{Q}_r(x,t))&\leq C\left(\frac{|E_{\lambda,\delta_2}\cap \tilde{Q}_r(x,t)|}{|\tilde{Q}_r(x,t)|}\right)^\nu w(\tilde{Q}_r(x,t))
      \\&\leq  C\left(c_{18} \left((\delta_3^2+\delta_1^2)+ (\delta_3^2+1+\delta_1^2)\delta_2\right)\right)^\nu w(\tilde{Q}_r(x,t))
      \\&< \varepsilon w(\tilde{Q}_r(x,t)),
 \end{align*}   
     where $\delta_3,\delta_1,\delta_2$ are appropriately chosen and  $(C,\nu)=[w]_{\mathbf{A}_\infty}$.\\  
    Therefore, for all $(x,t)\in \Omega_T$, $r\in (0,2r_0]$ and $\lambda>0$, if  
                $$w(E_{\lambda,\delta_2}\cap \tilde{Q}_r(x,t))\geq \varepsilon w(\tilde{Q}_r(x,t)),$$                
then $$ \tilde{Q}_r(x,t)\cap \Omega_T\subset F_\lambda,$$ 
         where $\delta_1=\delta_1(N,\Lambda_1,\Lambda_2,\varepsilon,[w]_{\mathbf{A}_\infty})\in (0,1)$ and $\delta_2=\delta_2(N,\Lambda_1,\Lambda_2,\varepsilon,[w]_{\mathbf{A}_\infty},T_0/R_0)\in (0,1)$. Combining this with \eqref{5hh2310131}, we can apply Lemma \ref{5hhvitali2} to get the result.
              \end{proof}\medskip\\
\begin{proof}[Proof of Theorem \ref{101120143}]By Theorem \ref{5hh23101312}, for any $\varepsilon>0,R_0>0$ one finds  $\delta=\delta_1(N,\Lambda_1,\Lambda_2,\varepsilon,[w]_{\mathbf{A}_\infty})\in (0,1/2)$ and $\delta_2=\delta_2(N,\Lambda_1,\Lambda_2,\varepsilon,[w]_{\mathbf{A}_\infty},T_0/R_0)\in (0,1)$ and $\Lambda=\Lambda(N,\Lambda_1,\Lambda_2)>0,s_0=s_0(N,\Lambda_1,\Lambda_2)$ such that if $\Omega$ is  a $(\delta,R_0)$- Reifenberg flat domain and $[A]_{s_0}^{R_0}\le \delta$ then 
   \begin{equation}\label{1411201413}
   w(\{\mathcal{M}(|\nabla u|^2)>\Lambda\lambda, \mathcal{M}[|F|^2]\le \delta_2\lambda \}\cap \Omega_T)\le B\varepsilon w(\{ \mathcal{M}(|\nabla u|^2)> \lambda\}\cap \Omega_T),
   \end{equation}
   for all $\lambda>0$, 
      where the constant $B$  depends only on $N,\Lambda_1,\Lambda_2, T_0/R_0, [w]_{\mathbf{A}_\infty}$.
      Thus, for $s<\infty,$
      \begin{align*}
 ||\mathcal{M}(|\nabla u|^2)||_{L^{q,s}_w(\Omega_T)}^s&
 =q\Lambda^s\int_{0}^{\infty}\lambda^s\left(w(\{\mathcal{M}(|\nabla u|^2)>\Lambda\lambda \}\cap \Omega_T)\right)^{s/q}\frac{d\lambda}{\lambda} \\&\leq 
q\Lambda^s2^{s/q}(B\varepsilon)^{s/q}\int_{0}^{\infty}\lambda^s\left(w(\{\mathcal{M}(|\nabla u|^2)>\lambda \}\cap \Omega_T)\right)^{s/q}\frac{d\lambda}{\lambda} 
\\&+  q\Lambda^s2^{s/q}\int_{0}^{\infty}\lambda^s\left(w(\{\mathcal{M}(|F|^2)>\delta_2\lambda \}\cap \Omega_T)\right)^{s/q}\frac{d\lambda}{\lambda}
\\& = \Lambda^s2^{s/q}(B\varepsilon)^{s/q}||\mathcal{M}(|\nabla u|^2)||_{L^{q,s}_w(\Omega_T)}^s+\Lambda^s2^{s/q}\delta_2^{-s}||\mathcal{M}(|F|^2)||_{L^{q,s}_w(\Omega_T)}^s.
      \end{align*}
It implies
\begin{align*}
||\mathcal{M}(|\nabla u|^2)||_{L^{q,s}_w(\Omega_T)}\leq 2^{1/s}\Lambda2^{1/q}(B\varepsilon)^{1/q}||\mathcal{M}(|\nabla u|^2)||_{L^{q,s}_w(\Omega_T)}+2^{1/s}\Lambda 2^{1/q}\delta_2^{-1}||\mathcal{M}(|F|^2)||_{L^{q,s}_w(\Omega_T)}
\end{align*}
and this inequalities  is also true when $s=\infty$. \\
We can choose $\varepsilon=\varepsilon(N,\Lambda,s,q,B)>0$ such that   $2^{1/s}\Lambda2^{1/q}(B\varepsilon)^{1/q}\leq 1/2$, then we get the result.
\end{proof}\medskip\\                              
\begin{proof}[Proof of Theorem \ref{161120141}] We recall that $A(x,t,\xi)=A(x,t)\xi$ where $A(x,t)$ is a matrix.\\
\textbf{a.}  Fix $q>2, 0<s\leq \infty$, $w\in \mathbf{A}_{q/2}$. Assume $|||F|||_{L^{q,s}_w(\Omega_T)}<\infty$. So, $F\in L^2(\Omega_T,\mathbb{R}^N)$ and problem \eqref{161120143} with $\mu\equiv0,\sigma\equiv 0$ has a unique weak solution $v_1\in L^2(0,T,H_0^{1}(\Omega))$. By Theorem \ref{101120141}, 
we find a $\delta_1=\delta_1(N,\Lambda_1,\Lambda_2,q,s, [w]_{\mathbf{A}_{q/2}})\in (0,1)$ such that if $\Omega$ is  a $(\delta_1,R_0)$-Reifenberg flat domain  and $[A]_{s_0}^{R_0}\le \delta_1$ for some $R_0>0$ then                     
                                                                                                                                    \begin{equation}\label{161120144**}
                                                                                                                                    |||\nabla v_1|||_{L^{q,s}_w(\Omega_T)}\leq c_1 |||F|||_{L^{q,s}_w(\Omega_T)},
                                                                                                                                    \end{equation} 
                                                                  where $c_1=c_1(N,\Lambda_1,\Lambda_2,q,s, [w]_{\mathbf{A}_{q/2}},T_0/R_0)$.\\
                                                                  Moreover,  by \cite[Theorem 2.20]{55QH2}, there exists a distribution  solution $v_2\in L^1(0,T,W_0^{1,1}(\Omega))$ of \eqref{161120143} with $F\equiv 0$ and  $\delta_2=\delta_2(N,\Lambda_1,\Lambda_2,q,s, [w]_{\mathbf{A}_{q/2}})\in (0,1)$ such that if $\Omega$ is  a $(\delta_2,R_0)$-Reifenberg flat  and $[A]_{s_0}^{R_0}\le \delta_2$ for some $R_0>0$, then there holds                            
                                                                                                                                 \begin{equation}\label{161120144*}
                                                                                                                                    |||\nabla v_2|||_{L^{q,s}_w(\Omega_T)}\leq c_2 ||\mathcal{M}_1[\omega]||_{L^{q,s}_w(\Omega_T)},
                                                                                                                                    \end{equation}
                                                                                                                                    where $c_2=c_2(N,\Lambda_1,\Lambda_2,q,s, [w]_{\mathbf{A}_{q/2}},T_0/R_0)$. In particular, $v_2\in L^2(0,T,H_0^1(\Omega))$.  \\
  Obviously, $u:=v_1+v_2$ is a unique weak solution of \eqref{161120143} in $L^2(0,T,H_0^1(\Omega))$ and from \eqref{161120144**}-\eqref{161120144*} we obtain \eqref{161120144} where $\Omega$ is a $(\delta,R_0)$-flat  and $[A]_{s_0}^{R_0}\le \delta$ with $\delta=\min\{\delta_1,\delta_2\}.$\\
\textbf{b.} Using the previous argument, we only show statement \textbf{b} in case $\mu\equiv 0,\sigma\equiv 0$. \\
Fix $\varepsilon\in (0,1), \frac{2(\varepsilon+1)}{\varepsilon+2}<q\leq 2, 0<s\leq \infty$, $w^{2+\varepsilon}\in \mathbf{A}_{1}$ and assume $\mathcal{M}_1[\omega], |F|\in L^{q,s}_w(\Omega_T)$. Set $p=\frac{2(\varepsilon+1)}{\varepsilon+2}$. \\
\textbf{b.1.}~  We prove  that there is a $\delta_3=\delta_3(N,\Lambda_1,\Lambda_2,\varepsilon)\in (0,1)$ such that if  $\Omega$ is a $(\delta,R_0)$-Reifenberg flat domain for some $R_0>0$, then problem \eqref{161120143} with $\mu\equiv0,\sigma\equiv 0$ has a unique weak solution $v_3\in L^p(0,T,W_0^{1,p}(\Omega))$.\\  Clearly, if $A^{*}(x,t,\xi)=A^{*}(x,t)\xi$, where $A^{*}(x,t)$ is the transposed matrix of $A(x,t)$ then $A^*$ satisfies \eqref{5hhconda} and \eqref{5hhcondb} with the same constants and  $[A^{*}]_{s_0}^{R_0}=[A]_{s_0}^{R_0}$.\\
By Theorem \ref{101120141} there exists $\delta_3=\delta_3(N,\Lambda_1,\Lambda_2,\varepsilon)\in (0,1)$ such that if $\Omega$ is  $(\delta_3,R_0)$-flat  and $[A]_{s_0}^{R_0}\le \delta_3$ for some $R_0>0$ there holds 
\begin{align}\label{171120144*}
|||\nabla \varphi |||_{L^{p'}(\Omega_T)}\leq c_3 |||G|||_{L^{p'}(\Omega_T)}~~\forall G\in C^\infty(\overline{\Omega}_T,\mathbb{R}^N),
\end{align}
for some constant $c_3$, where $\varphi$ is a unique solution to the problem 
\begin{equation}\label{171120145}
                       \left\{
                                       \begin{array}
                                       [c]{l}%
                                       -{\varphi_{t}}-\operatorname{div}(A^{*}(x,t)\nabla \varphi)=\operatorname{div}(G)~~\text{in }\Omega_T,\\ 
                         u=0~~~~~~~\text{on}~~
                                                                                              \partial \Omega \times (0,T),\\
                                                               u(T)=0~~\text{in }~\Omega. 
                                                                                                 \\                          
                                       \end{array}
                                       \right.  
                                       \end{equation}
   Let $F_n\in C_c^\infty(\Omega_T,\mathbb{R}^N)$ converge to $F$ in $L^p(\Omega_T,\mathbb{R}^N)$ and $u_n$ be a solution of problem \eqref{161120143} with $F=F_n$ and  $\mu\equiv0,\sigma\equiv 0$. We can choose  $\varphi$ for test function,
  \begin{align*}
  -\int_{\Omega_T}\nabla u_n G dxdt&=-\int_{\Omega_T}u_n\varphi_tdxdt+\int_{\Omega_T}A^{*}(x,t)\nabla\varphi\nabla u_ndxdt\\& =-\int_{\Omega_T}u_n\varphi_tdxdt+\int_{\Omega_T}A(x,t)\nabla u_n\nabla\varphi dxdt\\&=-
  \int_{\Omega_T}\nabla\varphi F_n.
  \end{align*} 
  Using H\"older inequality and \eqref{171120144*} yield                                 
 \begin{align*}
 |\int_{\Omega_T}\nabla u_n G dxdt|\leq c_3 |||G|||_{L^{p'}(\Omega_T)}|||F_n|||_{L^{p}(\Omega_T)}~~\forall G\in C^\infty(\overline{\Omega}_T,\mathbb{R}^N),
 \end{align*}
 it implies 
 \begin{align*}
 |||\nabla u_n|||_{L^{p}(\Omega_T)}\leq c_3|||F_n|||_{L^{p}(\Omega_T)}.
 \end{align*}
 By linearity of $A$ we get 
  \begin{align*}
   |||\nabla u_n-\nabla u_m|||_{L^{p}(\Omega_T)}\leq c_3|||F_n-F_m|||_{L^{p}(\Omega_T)}\to 0 ~\text{ as}~ n,m\to \infty.
   \end{align*}                                     
 Thus, $u_n$ converges to some function $v_3$ in $L^{p}(0,T,W_0^{1,p}(\Omega))$. Obviously, $v_3$ is  a unique weak solution in $L^{p}(0,T,W_0^{1,p}(\Omega))$ of problem \eqref{161120143} with $\mu\equiv0,\sigma\equiv 0$.\medskip \\  
\textbf{b.2.}  Set $\overline{w}(x,t)=(w(x,t))^{-\frac{1}{p-1}}$. We have $\overline{w}\in \mathbf{A}_{p'/2}$ and $[\overline{w}]_{\mathbf{A}_{p'/2}}= [w^{2+\varepsilon}]_{\mathbf{A}_{1+\varepsilon}}^{\frac{1}{\varepsilon}}\leq  [w^{2+\varepsilon}]_{\mathbf{A}_{1}}^{\frac{1}{\varepsilon}}$. By Theorem \ref{101120141} there exists $\delta_4=\delta_4(N,\Lambda_1,\Lambda_2,\varepsilon, [w^{2+\varepsilon}]_{\mathbf{A}_{1}})\in (0,1)$ such that if $\Omega$ is  $(\delta_4,R_0)$-flat  and $[A]_{s_0}^{R_0}\le \delta_4$ for some $R_0>0$ there holds 
\begin{align}\label{171120144}
|||\nabla \varphi |||_{L^{p'}_{\overline{w}}(\Omega_T)}\leq c_4 |||G|||_{L^{p'}_{\overline{w}}(\Omega_T)}~~\forall G\in C^\infty(\overline{\Omega}_T,\mathbb{R}^N),
\end{align}
where $\varphi$ is a unique solution to problem \eqref{171120145} and $c_4=c_4(N,\Lambda_1,\Lambda_2,\varepsilon, [w^{2+\varepsilon}]_{\mathbf{A}_{1}})$. \\
Using 
$
\int_{\Omega_T}\nabla u G dxdt=
\int_{\Omega_T}\nabla\varphi F
$, H\"older inequality and \eqref{171120144} we find
\begin{align*}
|\int_{\Omega_T}\nabla u G dxdt|\leq c_4 |||F|||_{L^{p}_w(\Omega_T)}|||G|||_{L^{p'}_{\overline{w}}(\Omega_T)}~~\forall G\in C^\infty(\overline{\Omega}_T,\mathbb{R}^N).
\end{align*}
Thus, we obtain 
\begin{align*}
|||\nabla u|||_{L^{p}_w(\Omega_T)}\leq c_4 |||F|||_{L^{p}_w(\Omega_T)}.
\end{align*}
On the other hand, by statement \textbf{a} there exist $\delta_5=\delta_5(N,\Lambda_1,\Lambda_2, [w^{2+\varepsilon}]_{\mathbf{A}_{1}})\in (0,1)$ such that if $\Omega$ is  $(\delta_5,R_0)$-flat  and $[A]_{s_0}^{R_0}\le \delta_5$ for some $R_0>0$ there holds
\begin{align*}
|||\nabla u|||_{L^{3}_w(\Omega_T)}\leq c_5 |||F|||_{L^{3}_w(\Omega_T)}.
\end{align*} 
for some $c_5=c_5(N,\Lambda_1,\Lambda_2,\varepsilon, [w^{2+\varepsilon}]_{\mathbf{A}_{1}})$.
We now denote map $\mathcal{J}:(L^{p}_w(\Omega_T))^N\to L^{p}_w(\Omega_T)$ by $\mathcal{J}(f):=|\nabla v|$ for any $f \in (L^{p}_w(\Omega_T))^N$ where $v$ is the unique weak solution of problem \eqref{161120143} with $\mu\equiv0,\sigma\equiv 0$ and $F=f$. We see that $\mathcal{J}$ is a sublinear operator and 
\begin{align*}
||\mathcal{J}(f_1)||_{L^{3}_w(\Omega_T)}\leq c_5 |||f_1|||_{L^{3}_w(\Omega_T)}~~\forall~f_1\in~(L^{3}_w(\Omega_T))^N 
\end{align*} 
and 
\begin{align*}
||\mathcal{J}(f_2)||_{L^{p}_w(\Omega_T)}\leq c_4 |||f_2|||_{L^{p}_w(\Omega_T)}~~\forall~f_2\in~(L^{p}_w(\Omega_T))^N
\end{align*}
where $\Omega$ is  $(\delta,R_0)$-Reifenberg flat  and $[A]_{s_0}^{R_0}\le \delta$ with $\delta=\min\{\delta_4,\delta_5\}$.
Thank to the interpolation Theorem, see \cite[Theorem 1.4.19]{55Gra} we get the statement \textbf{b}. This completes the proof. 
\end{proof}                                  
                              
             \section{Quasilinear Riccati type parabolic equations}  
 
     To prove Theorem \ref{5hh260320144} we need the following Lemma: 
     \begin{lemma}\label{171120146}Let $ \gamma\geq 1$ and $H_1,H_2$ be measurable functions in $\mathbb{R}^N$. If 
     \begin{align*}
     \int_{\Omega_T}|H_1|wdxdt\leq C(\gamma, [w^{\gamma}]_{\mathbf{A}_{1}})\int_{\Omega_T}|H_2|wdxdt~~\forall ~ w^{\gamma}\in \mathbf{A}_1,
     \end{align*}
     then for any $p>\frac{(N+2)(\gamma-1)}{\gamma}$, 
    \begin{align*}                              [|H_1|]_{\operatorname{Cap}_{\mathcal{G}_1,p}} \leq  C  [|H_2|]_{\operatorname{Cap}_{\mathcal{G}_1,p}},                                                          
                                                 \end{align*}
            where $C=C(N,p,\gamma,T_0)$, for measurable function $H$ in $\mathbb{R}^N$,  $[H]_{\operatorname{Cap}_{\mathcal{G}_1,p}} $ is denoted by 
            \begin{align*}
            [|H|]_{\operatorname{Cap}_{\mathcal{G}_1,p}} =\sup\frac{\int_{K}|H|dxdt}{\operatorname{Cap}_{\mathcal{G}_1,p}(K)}, 
            \end{align*}the suprema being taken over all compact sets $K\subset\mathbb{R}^{N+1}$.  
     \end{lemma}
     Its proof can be found in \cite[Proof of Propostion 4.24]{55QH2}. Using this Lemma we obtain
     \begin{theorem} \label{1711201410}  Suppose that $A$ is linear. Let $F\in L^1(\Omega_T,\mathbb{R}^N), \mu\in\mathfrak{M}_b(\Omega_T), \sigma\in\mathfrak{M}_b(\Omega)$, set $\omega=|\mu|+|\sigma|\otimes\delta_{\{t=0\}}$. Let $s_0$ be in Theorem \ref{101120143}.
        \begin{description}
        \item[a.] For any $q>1$ and $\mathcal{M}_1[\omega],\mathcal{M}(|F|^2)^{1/2}\in L^q(\Omega_T)$ we find  $\delta=\delta(N,\Lambda_1,\Lambda_2,q)\in (0,1)$ such that if $\Omega$ is a $(\delta,R_0)$-Reifenberg flat domain  and $[A]_{s_0}^{R_0}\le \delta$ for some $R_0>0$ then  there exists a unique weak solution $u\in  L^q(0,T,W_0^{1,q})$ of \eqref{161120143} and there holds                     
                                                                      \begin{align}\label{111120142a}                                    
                                                                                          \left[|\nabla u|^q\chi_{\Omega_T}\right]_{\operatorname{Cap}_{\mathcal{G}_1,q'}}\leq C_1 \left[(\mathcal{M}(|F|^2))^{q/2}\chi_{\Omega_T}\right]_{\operatorname{Cap}_{\mathcal{G}_1,q'}}+ C_1\left[\omega\right]_{\operatorname{Cap}_{\mathcal{G}_1,q'}}^q                             
                                                                                                                  \end{align}
    where $C_1=C_1(N,\Lambda_1,\Lambda_2,q,T_0/R_0,T_0)$.
           \item[b.] For any $1<q<\frac{N+2}{N}$ and $\mathcal{M}_1[\omega], |F|\in L^{q}(\Omega_T)$ we find  $\delta=\delta(N,\Lambda_1,\Lambda_2,q)\in (0,1)$ such that if $\Omega$ is a $(\delta,R_0)$-flat domain  and $[A]_{s_0}^{R_0}\le \delta$ for some $R_0>0$ then there exists a unique weak solution $u\in  L^{q}(0,T,W_0^{1,q}(\Omega))$ of \eqref{161120143} and there holds                            
                                                                                                                                             \begin{align}\label{111120142b}                                    
                                                                                                                                                                                                                                        \left[|\nabla u|^q\chi_{\Omega_T}\right]_{\operatorname{Cap}_{\mathcal{G}_1,q'}}\leq C_2 \left[|F|^q\chi_{\Omega_T}\right]_{\operatorname{Cap}_{\mathcal{G}_1,q'}}+ C_2\left[\omega\right]_{\operatorname{Cap}_{\mathcal{G}_1,q'}}^q                                 
                                                                                                                                                                                                                                                               \end{align}
      where  $C_2=C_2(N,\Lambda_1,\Lambda_2,q,T_0/R_0,T_0)$.
                            
        \end{description}                  
                                                                                                                   \end{theorem}
                                                                                                                   \begin{proof} We have 
\begin{align*}
\left[(\mathcal{M}_1[\omega])^q\chi_{\Omega_T}\right]_{\operatorname{Cap}_{\mathcal{G}_1,q'}} \leq c \left[\omega\right]_{\operatorname{Cap}_{\mathcal{G}_1,q'}}^q,
\end{align*}                                                where $c=c(N,q,T_0)$, see \cite[Corollary 4.39]{55QH2}. 
Therefore, thanks to Theorem \ref{101120143}, \ref{161120141} and Lemma \ref{171120146}   we get the results.                                          
   \end{proof}\medskip\\             
            \begin{proof}[Proof of Theorem \ref{5hh260320144}]
            By Theorem  \ref{1711201410}, there exists   $\delta=\delta(N,\Lambda_1,\Lambda_2,q)\in (0,1)$  such that $\Omega$ is $(\delta,R_0)$- Reifenberg flat domain  and $[\mathcal{A}]_{s_0}^{R_0}\le \delta$ for some $R_0$ and a sequence
            $\{u_{n}\}_{n}$ 
           obtained by induction of the  weak solutions of
            \begin{align*}
             \left\{
                        \begin{array}
                        [c]{l}%
                        (u_{1})_{t}-\operatorname{div} (A(x,t,\nabla u_1))=\operatorname{div}(F)+\mu~~\text{in }\Omega_T,\\
                         u_1 = 0\quad \text{ on }~ \partial\Omega\times (0,T), \\ 
                         u_1(0)=\sigma~~\text{ in }~\Omega,
                        \end{array}
                        \right.
            \end{align*}           
            and 
              \begin{align*}
                         \left\{
                                    \begin{array}
                                    [c]{l}%
                                    (u_{n+1})_{t}-\operatorname{div} (A(x,t,\nabla u_{n+1}))=|\nabla u_n|^q+\operatorname{div}(F)+\mu~~\text{in }\Omega_T,\\
                                     u_{n+1} = 0\quad \text{ on }~ \partial\Omega\times (0,T), \\ 
                                     u_{n+1}(0)=\sigma~~\text{ in }~\Omega,
                                    \end{array}
                                    \right.
                        \end{align*}                             for any which satisfy
                        \begin{equation}\label{5hh060920140}
                        \left[|\nabla u_{n+1}|^q\chi_{\Omega_T}\right]_{\operatorname{Cap}_{\mathcal{G}_1,q'}}\leq c \left[H_q\right]_{\operatorname{Cap}_{\mathcal{G}_1,q'}}+c\left[|\nabla u_{n}|^q\chi_{\Omega_T}\right]^q_{\operatorname{Cap}_{\mathcal{G}_1,q'}}+ c\left[\omega\right]_{\operatorname{Cap}_{\mathcal{G}_1,q'}}^q ~~\forall n\geq 0, 
                        \end{equation}
                        where $u_0\equiv0$ and  the constant $c$ depends only on $N,\Lambda_1,\Lambda_2,q$ and $T_0/R_0,T_0$. 
                        Since, $u_{n+1}-u_{n}$ is the unique weak solution of 
                                                \begin{equation}\left\{ \begin{array}{l}
                                                                  {u_t} - \operatorname{div}\left( {A(x,t,\nabla u)} \right) = |\nabla u_{n}|^q-|\nabla u_{n-1}|^q~\text{in }~\Omega_T,   \\ 
                                                                  u = 0\quad \quad \text{ on }~~\partial\Omega\times(0,T), \\
                                                                  u(0)=0 \quad \text{ in } \Omega, 
                                                                  \end{array} \right.\end{equation}
 we have
\begin{equation}\label{5hh060920140*}
                        \left[|\nabla u_{n+1}-\nabla u_n|^q\chi_{\Omega_T}\right]_{\operatorname{Cap}_{\mathcal{G}_1,q'}}\leq c\left[||\nabla u_{n}|^q-|\nabla u_{n-1}|^q|\chi_{\Omega_T}\right]^q_{\operatorname{Cap}_{\mathcal{G}_1,q'}}~~\forall n\geq 0. 
                        \end{equation}          
    If
      \begin{align}\label{1711201412}
      \left[H_q\right]_{\operatorname{Cap}_{\mathcal{G}_1,q'}}+\left[\omega\right]_{\operatorname{Cap}_{\mathcal{G}_1,q'}}^q\leq \frac{q-1}{(cq)^{q'}},
      \end{align}
   then from \eqref{5hh060920140} we can show that 
       \begin{align}\label{1711201411}
        \left[|\nabla u_{n}|^q\chi_{\Omega_T}\right]_{\operatorname{Cap}_{\mathcal{G}_1,q'}}\leq cq'\left(\left[H_q\right]_{\operatorname{Cap}_{\mathcal{G}_1,q'}}+\left[\omega\right]_{\operatorname{Cap}_{\mathcal{G}_1,q'}}^q\right)~~\forall n\geq 1.
       \end{align}
   Using H\"older inequality and  \eqref{5hh060920140*} yield 
  \begin{align*}
  &\left[|\nabla u_{n+1}-\nabla u_n|^q\chi_{\Omega_T}\right]_{\operatorname{Cap}_{\mathcal{G}_1,q'}}\leq cq\left[||\nabla u_{n}-\nabla u_{n-1}|(|\nabla u_{n}|^{q-1}+|\nabla u_{n-1}|^{q-1})\chi_{\Omega_T}\right]^q_{\operatorname{Cap}_{\mathcal{G}_1,q'}}\\&~~~~~\leq 
  cq\left[|\nabla u_{n}-\nabla u_{n-1}|^q\chi_{\Omega_T}\right]_{\operatorname{Cap}_{\mathcal{G}_1,q'}}\left[(|\nabla u_{n}|^{q-1}+|\nabla u_{n-1}|^{q-1})^{q'}\chi_{\Omega_T}\right]_{\operatorname{Cap}_{\mathcal{G}_1,q'}}^{q-1}
  \\&~~~~~\leq 
    cq2^{q'-1}\left[|\nabla u_{n}-\nabla u_{n-1}|^q\chi_{\Omega_T}\right]_{\operatorname{Cap}_{\mathcal{G}_1,q'}} \left( \left[|\nabla u_{n}|^q\chi_{\Omega_T}\right]_{\operatorname{Cap}_{\mathcal{G}_1,q'}}^{q-1}+\left[|\nabla u_{n-1}|^q\chi_{\Omega_T}\right]_{\operatorname{Cap}_{\mathcal{G}_1,q'}}^{q-1} \right).
  \end{align*} 
  Hence, by \eqref{1711201412}-\eqref{1711201411} we find                                            
  \begin{align*}
    &\left[|\nabla u_{n+1}-\nabla u_n|^q\chi_{\Omega_T}\right]_{\operatorname{Cap}_{\mathcal{G}_1,q'}}\\&~~~~~~\leq 
      cq2^{q'}(cq')^{q-1}\left(\left[H_q\right]_{\operatorname{Cap}_{\mathcal{G}_1,q'}}+\left[\omega\right]_{\operatorname{Cap}_{\mathcal{G}_1,q'}}^q\right)^{q-1}\left[|\nabla u_{n}-\nabla u_{n-1}|^q\chi_{\Omega_T}\right]_{\operatorname{Cap}_{\mathcal{G}_1,q'}} 
      \\&~~~~~~\leq \frac{1}{2}\left[|\nabla u_{n}-\nabla u_{n-1}|^q\chi_{\Omega_T}\right]_{\operatorname{Cap}_{\mathcal{G}_1,q'}},
    \end{align*}
provided that\begin{align}\label{1711201413}
     \left[H_q\right]_{\operatorname{Cap}_{\mathcal{G}_1,q'}}+\left[\omega\right]_{\operatorname{Cap}_{\mathcal{G}_1,q'}}^q\leq\min\left\{ (cq'2^{q'+1}(cq')^{q-1})^{-\frac{1}{q-1}}, \frac{q-1}{(cq)^{q'}}\right\}.
    \end{align}
      Hence, if \eqref{1711201413} holds,  $ u_n$ converges to $u = u_1+\sum_{n=1}^{\infty}(u_{n+1}-u_n)$ in $L^q(0,T,W^{1,q}_0(\Omega))$ satisfying 
       \begin{equation*}
       \left[|\nabla u|^q\chi_{\Omega_T}\right]_{\operatorname{Cap}_{\mathcal{G}_1,q'}}\leq cq'\left(\left[H_q\right]_{\operatorname{Cap}_{\mathcal{G}_1,q'}}+\left[\omega\right]_{\operatorname{Cap}_{\mathcal{G}_1,q'}}^q\right).
       \end{equation*}                                     Obviously, $u$ is a weak solution of problem \eqref{111120143}. This completes the proof.                           
                                   \end{proof} \medskip\\      \textbf{Acknowledgements:} The author wishes to express his deep gratitude to Professor Laurent V\'eron and Professor Marie-Fran\c{c}oise Bidaut-V\'eron for encouraging, taking care and giving many useful comments during the preparation of the paper.  Besides,  the author would like to thank the anonymous referee for giving many remarks about the papers  \cite{CaPe,Iwa,Mi2,55Mi0}.                            

\end{document}